\documentclass[12pt, reqno]{amsart}
\usepackage{amsmath, amsthm, amscd, amsfonts, amssymb, graphicx, color, mathrsfs}
\usepackage[bookmarksnumbered, colorlinks, plainpages]{hyperref}
\usepackage[all]{xy}
\usepackage{slashed}
\usepackage{tipa}
\usepackage{mathabx}
\usepackage{soul}
\usepackage{cancel}
\usepackage{ulem}
\usepackage{inputenc}
\usepackage{booktabs}

\newtheorem{theorem}{Theorem}[section]
\newtheorem{lemma}[theorem]{Lemma}

\newtheorem{proposition}[theorem]{Proposition}
\newtheorem{corollary}[theorem]{Corollary}
\theoremstyle{definition}
\newtheorem{definition}[theorem]{Definition}
\newtheorem{example}[theorem]{Example}

\theoremstyle{remark}
\newtheorem{remark}[theorem]{Remark}
\numberwithin{equation}{section}
\newtheorem{assump}[theorem]{Assumption}

\def\ind{{\mathcal I}}
\def\sL{{\star_{L}}}
\def\sLs{\widetilde{\star}_{L}}

\newcounter{quotecount}

\begin{document}
\setcounter{page}{1}

\title[  Expansion of traces on  manifolds with boundary  ]{  Expansion of traces and Dixmier traceability  for global pseudo-differential operators on  manifolds with boundary}

\author[D. Cardona S\'anchez]{Duv\'an Cardona }
\address{
  Duv\'an Cardona:
  \endgraf
  Department of Mathematics: Analysis, Logic and Discrete Mathematics
  \endgraf
  Ghent University, Belgium
  \endgraf
  {\it E-mail address} {\rm duvanc306@gmail.com, Duvan.CardonaSanchez@UGent.be}
  }
  
\author[V. Kumar]{Vishvesh Kumar}
\address{
  Vishvesh Kumar:
  \endgraf
  Department of Mathematics: Analysis, Logic and Discrete Mathematics
  \endgraf
  Ghent University, Belgium
  \endgraf
  {\it E-mail address} {\rm vishveshmishra@gmail.com, Kumar.Vishvesh@UGent.be}
  }

\author[M. Ruzhansky]{Michael Ruzhansky}
\address{
  Michael Ruzhansky:
  \endgraf
  Department of Mathematics: Analysis, Logic and Discrete Mathematics
  \endgraf
  Ghent University, Belgium
  \endgraf
 and
  \endgraf
  School of Mathematical Sciences
  \endgraf
  Queen Mary University of London
  \endgraf
  United Kingdom
  \endgraf
  {\it E-mail address} {\rm Michael.Ruzhansky@UGent.be}
  }

\author[N. Tokmagambetov]{Niyaz Tokmagambetov}
\address{
  Niyaz Tokmagambetov:
  \endgraf
  Department of Mathematics: Analysis, Logic and Discrete Mathematics
  \endgraf
  Ghent University, Belgium
  \endgraf
  {\it E-mail address} {\rm Niyaz.Tokmagambetov@UGent.be}
  }



\thanks{The authors are supported by the FWO Odysseus 1 grant G.0H94.18N: Analysis and Partial Differential Equations. MR is also supported in parts by the EPSRC grant
EP/R003025/2.}

\begin{abstract}
Given a smooth  manifold $M$ (with or without boundary), in this paper we study the regularisation of traces for the global pseudo-differential calculus  in the context of non-harmonic analysis. Indeed, using the global pseudo-differential calculus on manifolds (with or without boundary) developed in  \cite{RT2016}, the Calder\'on-Vaillancourt Theorem and the global  functional calculus in \cite{CardonaKumarRuzhanskyTokmagambetov2020II}, we determine the singularity orders in the regularisation of traces and the sharp regularity orders for the Dixmier traceability of  the global H\"ormander classes.  Our  analysis (free of coordinate systems) allows us to obtain non-harmonic analogues of several classical results arising from the microlocal analysis of regularised traces for pseudo-differential operators with symbols defined by localisations.
\end{abstract} \maketitle

\tableofcontents
\allowdisplaybreaks
\section{Introduction}
\subsection{Outline}
 One of the  main objectives of this paper is to study trace expansions for operators of the form $A\psi(tE),$ $t>0,$ where $\psi$ is positive smooth compactly supported function defined on $\mathbb{R}_0^+$ (or $\psi(\lambda):=e^{-t\lambda}$), $E$ is a positive $L$-elliptic global pseudo-differential operator and $A$ is a global pseudo-differential operator, using the global pseudo-differential calculus on manifolds with or without boundary developed in \cite{RT2016} (which is also stable under the complex functional  calculus as it was shown in  \cite{CardonaKumarRuzhanskyTokmagambetov2020II}) with additions in \cite{RT2017} and \cite{DRT2017}. Our second main objective is to study the Dixmier traceability of the $L$-elliptic global pseudo-differential operators. The main novelty of the used approach is that our analysis of traces is free of coordinate systems, presenting a new point of view, in comparison with other works on the subject, see   \cite{Seeley, Seeley3, Seeley2,Grubb,GrubbSeeley,GrubbSchrohe} and references therein for instance. 

The trace expansions play a very vital role in several areas of mathematics including spectral geometry, mathematical physics, index theory, etc. (see \cite{Gilkey, Lesch2010}). Many authors studied the trace expansions for the prototype examples of $\psi(t)$ like $e^{-t z}$ and $t^{-z}$ leading to interesting trace expansions like heat trace expansion and zeta functions. In this paper we will also study heat trace expansions. The study of trace expansion can be traced back to H\"ormander \cite{Hormander1968} and,   Duistermaat and Guillemin \cite{DG75} using the machinery of Fourier integral operators and require the fact that the Fourier transform of $\psi$ is compactly supported near $0$. Seeley \cite{Seeley, Seeley3, Seeley2}, Grubb \cite{Grubb}, Grubb and Seeley \cite{GrubbSeeley}, and Grubb and Schrohe \cite{GrubbSchrohe} investigated  trace expansions using the pseudo-differential calculus. Seeley's ideas are based upon the meromorphy in the complex parameter. Recently, Fischer \cite{Fischer1} investigated real trace expansion for the operator of the form $A\psi(t\mathcal{L}),$ where $A$ and $\mathcal{L}$ are classical pseudo-differential operators on a compact manifold with $\mathcal{L}$ elliptic. The main tool she used is the continuous inclusion of the functional calculus of $\mathcal{L}$ into the pseudo-differential calculus whose proof relies on the Helffer-Sj\"ostrand formula. In \cite{Fischer2}, she also studied traces expansions on the torus and on more general compact Lie groups using the global symbolic calculus developed by Ruzhansky and Turunen \cite{Ruz}. The first and third named authors have  investigated the traces expansion on compact Lie groups by using the notion of global symbol and global ellipticity for subelliptic pseudo-differential calculus \cite{CardonaRuzhansky2020I}. In addition to this, they have also investigated the Dixmier tracebility of subelliptic pseudo-differential operators on compact Lie groups. In this paper, we use the methods developed in \cite{CardonaRuzhansky2020I} by first and third author to establish our results suitably adopted to our setting.

The global pseudo-differential calculus on a compact Lie group $G$ can be thought as the calculus based on the eigenfunction expansion of the Laplacian on $G.$ Building upon this idea, last two authors \cite{RT2016} developed a global pseudo-differential calculus on compact manifolds with  boundary using the Fourier series associated with bi-orthogonal eigenfunction expansion (also known as non-harmonic analysis) of a model operator $L$ (need not be self-adjoint or elliptic) with discrete spectrum, see also \cite{DRT2017}. This calculus on manifolds with (possibly empty) boundary can be considered as a non-harmonic analogy of the pseudo-differential calculus constructed in \cite{Ruz} by using the Fourier analysis attached to every compact Lie group.   

\subsection{Main results}
The aim of this paper is to characterise the spectral asymptotics for the $L^2$-traces of pseudo-differential operators associated with boundary value problems, by using the notion of global symbol,  introduced in \cite{RT2016}. This calculus uses a global definition of symbols using the notion of non-harmonic analysis on smooth manifolds (see \cite{RT2016} for details). To illustrate our main results, let us recall that any pseudo-differential operator $A:C^\infty(M)\rightarrow C^\infty(M),$ where $A\in \Psi^{m}_{1,0}(M, loc)$ is of order $m$ and $M$ is a closed manifold, is of trace class if $m<-n,$ with $n:=\dim(M)$ the topological dimension of $M.$ Moreover, if $dx$ is the volume form on $M,$ and  if $K\in C^{\infty}(M)\otimes_{\pi}\mathscr{D}'(M) $ is the Schwartz kernel of $A,$ then
\begin{equation*}
     \textnormal{\bf {Tr}}(A)=\int\limits_{M}K_A(x,x)dx.
\end{equation*}
  
It is known that in the case where $m>-n,$ $A$ may not be of trace class.  In the study of the spectral properties of $A,$ one can consider two kinds of regularisation processes, one with real methods  and another with the complex methods by taking complex powers. We will only consider  real methods in this paper. In fact, we consider a positive elliptic pseudo-differential operator\footnote{in the sense of \cite{RT2016}.} $E$ on $M,$ (which can be closed or with boundary) and we study the function \begin{equation}\label{Heatkernel}
    f_{A,E}(t):=\textnormal{\bf {Tr}}(Ae^{-tE}),
\end{equation} 
as a function of $t>0,$ or we consider the trace 
\begin{equation}\label{Cutoff}
    \Tilde{f}_{A,E}(t):=\textnormal{\bf {Tr}}(A\psi(tE)),
\end{equation}
where $\psi\in C^\infty_0(M)$ is a non-negative function. In both cases, it is expected that the functions $f_{A,E}(t)$ and $\Tilde{f}_{A,E}(t)$  provide spectral information of the operator $A,$ as well as, the geometric information of the manifold $M$. This fact motivates one of the fundamental problems in geometric analysis: the study of the regularisation of traces.   In view of the McKean-Singer index formula for the Atiyah-Singer index theorem
\begin{equation}\label{index}
     \textnormal{\bf ind}(T)= g_{I,T^*T}(t)-g_{I,T T^*}(t),\quad t>0,
\end{equation} 
the problem of
determining the coefficients of the asymptotic expansions of \eqref{Heatkernel}, is the starting point in the theory of invariants of Gilkey \cite{Gilkey}. Certainly, explicit expressions and geometric densities for the right hand side of  \eqref{index} were constructed in the classical work of Atiyah, Bott and Patodi \cite{atiyabott2}. Because of the  recent methods in the theory of pseudo-differential operators, where the global notion of symbols have shown to be a versatile tool\footnote{developed for  compact Lie groups by the third author and Turunen in \cite{Ruz}, for arbitrary smooth manifolds (with or without boundary) in \cite{RT2016,RT2017}, for arbitrary graded Lie groups in \cite{FischerRuzhanskyBook}, for general locally compact Lie groups of type I in \cite{M1,M2}, and for sub-Riemannian structures on compact Lie groups  in \cite{CardonaRuzhansky2020I}.}, we study the spectral asymptotics for traces of global pseudo-differential operators on compact manifolds defined by the quantisation procedure in  \cite{RT2016,RT2017}. One of the novelties of the present work, is that the global structure of the symbols in \cite{RT2016,RT2017}, allows us to simplify the analysis of regularisation of  traces in comparison with the classical  methods of spectral geometry for the calculus of H\"ormander \cite{Hormander1985III} for closed manifolds, or the techniques employed in the calculus for pseudo-differential operators on manifolds with boundary (see Schrohe \cite{Schrohe},  Grubb \cite{Grubb}, Grubb and Schrohe \cite{GrubbSchrohe} and Scott \cite{Scott} just to mention a few).

For  a pseudo-differential operator $L$ of order $\nu>0$ (in the sense of H\"ormander) on the interior of ${M},$ in  \cite{RT2016,RT2017} the authors associated a pseudo-differential calculus for continuous linear operators on $C^\infty_{L}(M)=\bigcap\limits_{k=0}^{\infty}\textnormal{Dom}(L^k)$ using the nonharmonic Fourier analysis related with $L.$ The classes of operators associated to this calculus are denoted by $\Psi^{m}_{\rho,\delta}(M\times \mathcal{I}),$\footnote{where $\mathcal{I}$ is the set indexing the discrete spectrum $\{\lambda_\xi:\xi\in \mathcal{I}\}$ of $L.$} and associated to every $A:C^\infty_{L}(M)\rightarrow C^\infty_{L}(M),$ with $A\in \Psi^{m}_{\rho,\delta}(M\times \mathcal{I}),$ there exists a unique function (symbol) $\sigma_A: M\times \mathcal{I}\rightarrow \mathbb{C}, $ such that
\begin{equation}\label{QuantizationIntro}
Af(x)=\sum_{\xi\in\mathcal{I}}
u_{\xi}(x)\sigma_{A}(x, \xi) \widehat{f}(\xi).
\end{equation}
Here $u_{\xi},$ $\xi\in \mathcal{I},$ are the eigenfunctions of $L,$ and the Fourier coefficients $\widehat{f}(\xi)$ are defined in  Section \ref{Sect2}, as well as, the suitable condition requested on $L$ in order that the quantisation formula \eqref{QuantizationIntro} will be well defined. From the properties of this calculus, we can assume  the Weyl eigenvalue counting formula (see Lemma \ref{equivalence} and the condition (WL) in Section \ref{expansionshere}), $$N(\lambda):=|\{\xi\in \mathcal{I}:(1+|\lambda_\xi|^2)^\frac{1}{2\nu}\leq \lambda \}|\sim \lambda^{Q},\,\quad \lambda\rightarrow\infty.$$  We will prove that (see Theorem \ref{asymptotictracemultiplier2}) any   continuous linear operator $A:C^\infty_L(M)\rightarrow\mathcal{D}'_L(M)$ with symbol  $\sigma\in {S}^{m}_{\rho,\delta}(M\times   \mathcal{I})$, $m\in \mathbb{R},$  satisfies that\footnote{with $\mathcal{M}_q:=(1+L^{\circ}L)^{\frac{q}{2\nu}},$ and $L^\circ$ is the conjugate to $L,$ defined in \eqref{EQ:Lo-def}.}
\begin{equation}\label{asymp122Intro}
  |\textnormal{\bf {Tr}}   (Ae^{-t\mathcal{M}_q})|\leq  c_{m,Q} t^{-\frac{Q+m}{q}}\int\limits_{t^{\frac{1}{q}}}^{\infty}e^{-s^{q}}s^{Q+m-1}ds,\,\,\forall t>0.
\end{equation}
In particular, for $m=-Q,$ we have
\begin{equation}\label{asymp1212Intro}
   |\textnormal{\bf {Tr}}  (Ae^{-t\mathcal{M}_q})|\leq  -c_Q\frac{1}{q}\log(t),\,\,\forall t\in (0,1),
\end{equation}while for $m>-Q,$ we have the estimate
\begin{equation}\label{asymp1212'Intro}
 |\textnormal{\bf {Tr}}(Ae^{-t\mathcal{M}_q})|\leq  t^{-\frac{Q+m}{q}}\left(  \sum_{k=0}^{\infty}b_k't^{\frac{k}{q}}\right),\,\,t\rightarrow 0^{+}.
\end{equation}
Under similar hypothesis, we will prove in Theorem \ref{asymptotictraceeta}, that 
\begin{equation}\label{asymp122222Intro}
    |\textnormal{\bf{Tr}}(A\psi(t E))|\leq C_{Q} \frac{1}{q}\int\limits_{ t^{\frac{1}{q}}  }^\infty\psi(s) \frac{ds}{s},\,\,\forall t>0,
\end{equation}
provided that $\psi\in L^{1}(\mathbb{R}^{+}_0;\frac{ds}{s})\bigcap C^{\infty}_0(\mathbb{R}^{+}_0)$ is positive, and $m=-Q.$ On the other hand, for $m>-Q$ and $\psi\in C^{\infty}_{0}(\mathbb{R}^+_0)$ being positive, we have
\begin{equation}\label{asymp1212222Intro}
       |\textnormal{\bf{Tr}}(A\psi(t E))|\leq C_{m,Q} t^{-\frac{1}{q}(Q+m)}\frac{1}{q}\int\limits_{ t^{\frac{1}{q}}  }^\infty\psi(s)s^{\frac{Q+m}{q}}\times \frac{ds}{s},\,\,\forall t>0.
\end{equation}So, we have the estimate (see \eqref{asymp1212'22})
\begin{equation}\label{asymp1212'22Intro}
|\textnormal{\bf{Tr}}(A\psi(t E))|\leq t^{-\frac{Q+m}{q}}\left(  \sum_{k=0}^{\infty}a_kt^{\frac{k}{q}}\right),\,\,t\rightarrow 0^{+}
\end{equation}for $m> -Q.$  
\begin{remark}If $L$ admits a self-adjoint extension on $L^2(M),$ the estimates \eqref{asymp1212'Intro} and \eqref{asymp1212'22Intro} can be improved  to  asymptotic expansions, provided that the symbol  $\sigma\in {S}^{m}_{\rho,\delta}(M\times   \mathcal{I})$, $m\in \mathbb{R},$ of the continuous linear operator $A:C^\infty_L(M)\rightarrow\mathcal{D}'_L(M)$ will be  positive\footnote{this means, its symbol satisfies $\sigma(x,\xi)\geq 0$ for all $(x,\xi)\in M\times \mathcal{I}.$} and $L$-elliptic\footnote{that is, there exist constants $C_0>0$ and $N_0\in\mathbb N$ such that 
$  |\sigma(x,\xi)| \geq C_0 \langle\xi\rangle^m $
for all $(x,\xi)\in M\times\ind$ for which
$\langle \xi \rangle \geq N_0,$
where  $\langle\xi\rangle:=(1+|\lambda_\xi|^2)^\frac{1}{2\nu},$  $\xi\in \mathcal{I}.$}. For details see Theorem \ref{L*=L1} and Theorem \ref{L*=L2}. In particular, Theorem \ref{L*=L1} gives the asymptotic expansions, 
\begin{equation}\label{ResidueIntro}
   \textnormal{\bf {Tr}}  (Ae^{-t\mathcal{M}_q}) \sim  -c_Q\frac{1}{q}\log(t),\,\,\forall t\in (0,1),
\end{equation} and 
\begin{equation}\label{ExpansionIntro}
 \textnormal{\bf {Tr}}(Ae^{-t\mathcal{M}_q}) \sim C_{m,Q} t^{-\frac{Q+m}{q}}\left(  \sum_{k=0}^{\infty}b_k't^{\frac{k}{q}}\right),\,\,t\rightarrow 0^{+},
\end{equation} for $m>-Q.$ Additionally, note that   for a closed manifold  $M,$ and for $L=\sqrt{\Delta_M},$ with $\Delta_M$ being  the positive Laplacian on $M,$    $\partial M=\emptyset,$ and  $Q=\dim(M).$ Taking $\mathcal{M}_q:=(1+\Delta_M)^{\frac{q}{2}},$ and an  elliptic, positive and classical pseudo-differential operator $A$ with order $m,$ \eqref{ResidueIntro} and  \eqref{ExpansionIntro} recover the classical Pleijel type formula for the expansion of traces (see Atiyah, Bott, and  Patodi \cite{atiyabott2} and Remark \ref{remarkintro} below) while that the sharp version of \eqref{asymp1212'22Intro}  in \eqref{321},
$$ \textnormal{\bf{Tr}}(A\psi(t E)) \sim t^{-\frac{Q+m}{q}}\left(  \sum_{k=0}^{\infty}a_kt^{\frac{k}{q}}\right),\,\,t\rightarrow 0^{+} ,
  $$
recovers, as a special case, the asymptotic expansion of the trace $\textnormal{\bf{Tr}}(A\psi(t E))$ in \cite{Fischer1}, taking as above,  $E\equiv\mathcal{M}_q:=(1+\Delta_M)^{\frac{q}{2}},$ $q>0,$ $Q=\dim(M),$ and $A$ being positive, classical of order $m$ (see Remark \ref{remarkintro}) and elliptic on $M$. 
\end{remark}

The analysis in proving the asymptotic formulae above along with the use of the Tauberian
theorem of Hardy and Littlewood in noncommutative geometry \cite{Connes94} and the functional calculus developed in \cite{CardonaKumarRuzhanskyTokmagambetov2020II} by the authors,  also implies that the operators with positive symbols in  $$\Psi^{-Q}_{\rho,\delta}(M\times \mathcal{I}),\quad 0\leq\delta<\rho\leq 1,$$ belong to the Dixmier ideal on $L^2(M)$. The present paper will be dedicated to proving the aforementioned statements.
\begin{remark}\label{remarkintro}
To illustrate the expansions above,  let us recall an interesting situation that comes from the spectral geometry of pseudo-differential operators with symbols defined by localisations. Let us precise this idea in detail. 
For  a compact orientable manifold without boundary $M$ of dimension $\varkappa,$  a  pseudo-differential operator $A$ on $M$ can be defined by using the notion of a local symbol, this means that for any local chart $U$, the operator $A$ has the form $$Au(x)=\int_{T^{*}_xU}e^{2\pi ix\cdot \xi}\sigma^A(x,\xi)\widehat{u}(\xi)\, d\xi.$$ The pseudo-differential operator $A$ is called classical, if $\sigma^A$ admits an asymptotic expansion 
$\sigma^A(x,\xi)\sim \sum_{j=0}^{\infty}\sigma^A_{m-j}(x,\xi)$ in such a way that each function $\sigma_{m-j}(x,\xi)$ is homogeneous in $\xi$ of order $m-j$ for $\xi\neq 0$. The set of classical pseudo-differential operators of order $m$ is denoted by $\Psi^m_{cl}(M)$. For $A\in\Psi_{cl}^m(M)$, and for $x\in M$, $\int_{|\xi|=1}\sigma_{-\varkappa}(x,\xi)\, d\xi$ defines a local density which can be glued over $M$. In this case, the non-commutative residue of $A$ is defined by the expression
\begin{equation}
\textnormal{res}\,(A)=\frac{1}{\varkappa(2\pi)^\varkappa}\int_{M}\int_{| \xi|=1}\sigma^A_{-\varkappa}(x,\xi)\, d\xi\,dx.
\end{equation}
If $E$ is a positive elliptic pseudo-differential operator of order $q>0,$ for every elliptic and positive pseudo-differential operator $A$ with order $m,$ $m\geqslant -\varkappa,$ we have 
 \begin{equation}\label{subellipticWR}
     \textnormal{\textbf{Tr}}(Ae^{-tE})\sim t^{-\frac{m+\varkappa}{q}}\sum_{k=0}^\infty a_kt^{\frac{k}{q}}-\frac{b_0}{q}\log(t)+O(1).
 \end{equation}If $m>-\varkappa,$ $b_0=0,$ and for $m=-\dim(M),$ $a_{k}=0$ for every $k,$ and $b_0=\textnormal{res}(A)$ is the Wodzicki residue of $A,$ see e.g. Wodzicki \cite{Wodzicki} and Lesch \cite{Lesch}.
\end{remark}Certainly, the asymptotic expansions that we investigate in this work, are  non-harmonic analogues  of the spectral expansions in \cite{CardonaRuzhansky2020I} and of the ones investigated in \cite{Fischer1,Fischer2}.

\section{Preliminaries: global operators on compact manifolds with boundary}\label{Sect2} 
In this section we briefly present basics of pseudo-differential calculus in the context on non-harmonic analysis developed \cite{RT2016} (see also \cite{DRT2017, CardonaKumarRuzhanskyTokmagambetov2020II}).
\subsection{The operator $L,$ its Fourier analysis, and its distribution spaces}
\begin{assump}[The operator $L$]
\label{Assumption_1} The topological space $M=\overline{\Omega}$ denotes a smooth orientable  manifold with (possibly empty)  boundary $\partial \Omega.$ We assume that $M$ is orientable and endowed with a volume form $dx.$ The non-harmonic analysis (Fourier analysis on $M$) will be associated to a continuous linear operator $L$ on $\Omega$ with the following characteristics:\footnote{which are satisfied by a wide class of operators appearing in spectral geometry including: Dirichlet boundary value problems and Neumann boundary value problems for the Laplacian on manifolds with the smooth  boundary (including planar domains and Euclidean submanifolds), the Steklov problems, harmonic and anharmonic oscillators, elliptic operators on closed manifolds, and many other families of elliptic operators on manifolds with boundary. For this aspects and for a long list of examples we refer the reader to  \cite{RT2016} and references therein.}
\begin{itemize}
\item[(A1):]
    Here, $L$ is a pseudo-differential operator (need not be self-adjoint) of order $\nu>0$ on ${\Omega},$ defined in the sense of H\"ormander\footnote{This means that the symbol of $L$ defined by local coordinate systems is  in the Kohn-Nirenberg class $\Psi^\nu_{1,0}(U)$ for some (and hence for all) coordinate patch $\phi:U\rightarrow M,$ for details see \cite{Hormander1985III}.}. 
 We assume $L$ equipped with some boundary
conditions $\textnormal{(BC)}$ on $\partial \Omega,$ and the corresponding boundary value problem, at times, will be denoted by $L_\Omega$.  

\item[(A2):] We assume the condition (BC+) that the boundary conditions $\textnormal{(BC)}$ define  a Fr\'echet topological space, in such a way that with respect to the family of seminorms introduced in Definition   \ref{TestFunSp}, $\textnormal{Dom}(L)$ and $\textnormal{Dom}(L^*)$ are closed subspaces of 
\begin{equation} \label{testfunction}
    C^{\infty}_L(\overline{\Omega}):=\bigcap_{k\in \mathbb{N}}\textnormal{Dom}(L^k_{\Omega})\textnormal{  and  }C^{\infty}_{L^*}(\overline{\Omega}):=\bigcap_{k\in \mathbb{N}}\textnormal{Dom}((L_{\Omega}^{*})^{k}),
\end{equation}
respectively, where $L^*$ is the adjoint of $L$ on $L^2(\overline{\Omega})$ and $\textnormal{Dom}(L_\Omega^k)$ (analogously, $\textnormal{Dom}((L_\Omega^*)^k)$) is defined by $$
{\rm Dom}(L_\Omega^{k}):=\{f\in L^{2}(\overline{\Omega}): \,\,\, {
L}^{j}f\in {\rm Dom}(L_\Omega), \,\,\, j=0, \,1, \, 2, \ldots,
k-1\}\footnote{We note that in this way all operators ${L}^{k}$, $k\in\mathbb N$, are being equipped with the same boundary conditions (BC).}.
$$ 

\item[(A3):] Assume that ${L}_\Omega$ has   spectrum purely discrete
$\{\lambda_{\xi}\in\mathbb C: \, \xi\in\mathcal{I}\}$
on $L^{2}(\overline{\Omega})$,
and we order the eigenvalues with
the occurring multiplicities in the ascending order:
\begin{equation}\label{EQ: EVOrder}
 |\lambda_{j}|\leq|\lambda_{k}| \quad\textrm{ for } |j|\leq |k|.
\end{equation}

\item[(A3)':] Let $u_{\xi}$ be  the eigenfunction of ${L}_{\Omega}$ corresponding to the
eigenvalue $\lambda_{\xi}$ for each $\xi\in\mathcal{I}.$
The system
$\{u_{\xi}: \, \xi\in\mathcal{I}\}$ is a basis in $L^{2}(\overline{\Omega})$, i.e.
for every $f\in L^{2}(\overline{\Omega})$ there exists a unique series
$\sum_{\xi\in\mathcal{I}} a_\xi u_\xi(x)$ that converges to $f$ in  $L^{2}(\overline{\Omega})$.
  Here the eigenfunctions $u_\xi$ satisfy the boundary conditions
(BC).

    \item[(A4):] Define the operator
${L}^{\circ}$ by setting its values on the basis $u_{\xi}$ by
\begin{equation}\label{EQ:Lo-def}
{L}^{\circ} u_{\xi}:=\overline{\lambda_{\xi}} u_{\xi},\quad
\text{ for all } \xi\in\mathcal{I}.
\end{equation} 
Set $ 
\langle\xi\rangle:=(1+|\lambda_{\xi}|^2)^{\frac{1}{2\nu}}.$ The system $\{\langle\xi\rangle\}_{\xi\in \mathcal{I}}$
determines the set of eigenvalues of the positive (first order) operator
$({\rm I}+{L^\circ\, L})^{\frac{1}{2\nu}}.$ We assume that for some 
$s_0\in\mathbb R$ we have
\begin{equation}\label{Assumption_4}
    \sum_{\xi\in\mathcal{I}} \langle\xi\rangle^{-s_0}<\infty.
\end{equation} We will also note in Lemma \ref{equivalence} that \textnormal{(A4)} is equivalent to the fact that $L$ satisfies the Weyl-eigenvalue counting formula.\footnote{Also, note that \eqref{Assumption_4} is equivalent to the fact that
 $({\rm I}+{L^\circ L})^{-\frac{s_0}{4\nu}}$ is Hilbert-Schmidt.
Indeed,  $({\rm I}+{L^\circ L})^{-\frac{s_0}{4\nu}}$ is Hilbert-Schmidt if and only if,
\begin{equation}\label{EQ:HS-conv}
\|({\rm I}+{L^\circ L})^{-\frac{s_0}{4\nu}}\|_{\tt HS}^2\cong \sum_{\xi\in\mathcal{I}}
\langle\xi\rangle^{-s_0}<\infty.
\end{equation}}
\item[(WZ):] We assume (only in Section \ref{Dixmiersection}) that the eigenfunctions
of both $L$ and ${L}^*$ satisfy the WZ-condition (WZ stands for `without zeros'), in the sense that the functions $u_{\xi}(x), \,
v_{\xi}(x)$ do not have zeros in the domain $\overline{\Omega}$
for all $\xi\in\mathcal{I}$, and if there exist $C>0$
and $N\geq0$ such that
$$
\inf\limits_{x\in\overline{\Omega}}|u_{\xi}(x)|\geq
C\langle\xi\rangle^{-N},\,\,
\inf\limits_{x\in\overline{\Omega}}|v_{\xi}(x)|\geq
C\langle\xi\rangle^{-N},
$$
as $\langle\xi\rangle\to\infty$.
\end{itemize}
\end{assump}
\begin{remark}For the aspects of the non-harmonic analysis associated to $L$ without the WZ-condition, we refer the reader to \cite{RT2017}. In particular for Section \ref{expansionshere} we do not assume the WZ-condition.  However, as we will use in our further analysis on Dixmier traces, the global functional calculus and the Calder\'on-Vaillancourt Theorem established in \cite{CardonaKumarRuzhanskyTokmagambetov2020II}, we require the WZ-condition in Section \ref{Dixmiersection}. In particular, we also make use of the (differential) difference structure on $\mathcal{I}$ provided by the difference operators defined in Subsection \ref{SEC:differences}, which allow us to define the H\"ormander classes $S^{m}_{\rho,\delta}(\overline{\Omega}\times \mathcal{I})$ that were introduced in \cite{RT2016}.
\end{remark}
\begin{remark} Note that, the conjugate spectral problem is
$ {
L^{\ast}}v_{\xi}=\overline{\lambda}_{\xi}v_{\xi}$ in $ \overline{\Omega} $  for all $  \xi\in\mathcal{I},
$
which we equip with the conjugate boundary conditions which we may denote by (BC)$^*$.
This adjoint problem will be denoted by ${L}_\Omega^*$.
Let $\|u_{\xi}\|_{L^{2}}=1$ and $\|v_{\xi}\|_{L^{2}}=1$ for all
$\xi\in\mathcal{I}.$ Here, we can take biorthogonal systems
$\{u_{\xi}\}_{\xi\in\mathcal{I}}$ and $\{v_{\xi}\}_{\xi\in\mathcal{I}}$,
i.e.
\begin{equation}\label{BiorthProp}
(u_{\xi},v_{\eta})_{L^2}=0 \,\,\,\, \hbox{for}
\,\,\,\, \xi\neq\eta, \,\,\,\, \hbox{and} \,\,\,\,
(u_{\xi},v_{\eta})_{L^2}=1 \,\,\,\, \hbox{for} \,\,\,\, \xi=\eta,
\end{equation}
where
$ (f, g)_{L^{2}}:=\int\limits_{M}f(x)\overline{g(x)}dx$ is the 
inner product of  $L^{2}(\overline{\Omega})$. Note that from 
\cite{bari}, $\{v_{\xi}:  \xi\in\mathcal{I}\}$ is a basis in $L^{2}(\overline{\Omega}),$  because of (A3)'.
\end{remark}

\begin{definition}[Test functions associated to $L_{\Omega} $ and $L_{\Omega}^*$]\label{TestFunSp}
The space
$C_{{L}}^{\infty}(\overline{\Omega})$ defined by \eqref{testfunction} is called the
space of test functions for ${L}_\Omega$. 
The Fr\'echet topology of $C_{{L}}^{\infty}(\overline{\Omega})$ is given by the family of norms
\begin{equation}\label{EQ:L-top}
\|\varphi\|_{C^{k}_{{L}}}:=\max_{j\leq k}
\|{L}^{j}\varphi\|_{L^2(\overline{\Omega})}, \quad k\in\mathbb N_0,
\, \varphi\in C_{{L}}^{\infty}(\overline{\Omega}).
\end{equation}
Analogously to the ${L}$-case,  the space $C_{{
L^{\ast}}}^{\infty}(\overline{\Omega})$ corresponding to the adjoint operator ${L}_\Omega^*$ is defined as in \eqref{testfunction} and the family of seminorms can be defined by
\begin{equation}\label{EQ:L-top-adj}
\|\psi\|_{C^{k}_{{L}^*}}:=\max_{j\leq k}
\|({L}^*)^{j}\psi\|_{L^2(\overline{\Omega})}, \quad k\in\mathbb N_0,
\ \psi\in C_{{L}^*}^{\infty}(\overline{\Omega}).
\end{equation}

Since we have $u_\xi\in C^\infty_{{L}}(\overline{\Omega})$ and
$v_\xi\in C^\infty_{{L}^*}(\overline{\Omega})$ for all $\xi\in\mathcal{I}$, we observe  that the spaces
$C^\infty_{{L}}(\overline{\Omega})$ and $C^\infty_{{L}^*}(\overline{\Omega})$
are dense in $L^2(\overline{\Omega})$.
\end{definition}
In general, for functions $f\in C_{{L}}^{\infty}(\overline{\Omega})$ and
$g\in C_{{L}^*}^{\infty}(\overline{\Omega})$, the $L^2$-duality makes sense in view
of the formula
\begin{equation}\label{EQ:duality}
({L}f, g)_{L^2(\overline{\Omega})}=(f,{L}^*g)_{L^2(\overline{\Omega})}.
\end{equation}
Therefore, in view of the formula \eqref{EQ:duality},
it makes sense to define the distributions $\mathcal D'_{{L}}(\overline{\Omega})$
as the space which is dual to $C_{{L}^*}^{\infty}(\overline{\Omega})$.

\begin{definition}[Distributions associated to $L_\Omega$ and $L_{\Omega}^*$]\label{DistrSp}
The space $$\mathcal D'_{{
L}}(\overline{\Omega}):=\mathcal L(C_{{L}^*}^{\infty}(\overline{\Omega}),
\mathbb C)$$ of linear continuous functionals on
$C_{{L}^*}^{\infty}(\overline{\Omega})$ is called the space of
${L}$-distributions.
We can understand the continuity here either in terms of the topology
\eqref{EQ:L-top-adj} or in terms of sequences, see
Proposition \ref{TH: UniBdd}.
For
$w\in\mathcal D'_{{L}}(\overline{\Omega})$ and $\varphi\in C_{{L}^*}^{\infty}(\overline{\Omega})$,
we shall write
$$
w(\varphi)=\langle w, \varphi\rangle.
$$
For any $\psi\in C_{{L}}^{\infty}(\overline{\Omega})$,
$$
C_{{L}^*}^{\infty}(\overline{\Omega})\ni \varphi\mapsto\int\limits_{\overline{\Omega}}{\psi(x)} \, \varphi(x)\, dx
$$
is an ${L}$-distribution, which gives an embedding $\psi\in
C_{{L}}^{\infty}(\overline{\Omega})\hookrightarrow\mathcal D'_{{
L}}(\overline{\Omega})$.
We note that in the distributional notation formula \eqref{EQ:duality} becomes
\begin{equation}\label{EQ:duality-dist}
\langle{L}\psi, \varphi\rangle=\langle \psi,\overline{{L}^* \overline{\varphi}}\rangle.
\end{equation}
\end{definition}

With the topology on $C_{{L}}^{\infty}(\overline{\Omega})$
defined by \eqref{EQ:L-top},
the space $$\mathcal
D'_{{L^{\ast}}}(\overline{\Omega}):=\mathcal L(C_{{L}}^{\infty}(\overline{\Omega}), \mathbb C)$$
of linear continuous functionals on $C_{{L}}^{\infty}(\overline{\Omega})$
is called the
space of ${L^{\ast}}$-distributions.

\begin{proposition}\label{TH: UniBdd}
A linear functional $w$ on
$C_{{L}^*}^{\infty}(\overline{\Omega})$ belongs to $\mathcal D'_{{
L}}(\overline{\Omega})$ if and only if there exists a constant $c>0$ and a
number $k\in\mathbb N_0$ with the property
\begin{equation}
\label{EQ: UnifBdd-s1} |w(\varphi)|\leq
c \|\varphi\|_{C^{k}_{{L}^*}} \quad \textrm{ for all } \, \varphi\in C_{{
L}^*}^{\infty}(\overline{\Omega}).
\end{equation}
\end{proposition}
\begin{remark}
Suppose that for a linear continuous operator
$D:C_{{L}}^{\infty}(\overline{\Omega})\to C_{{L}}^{\infty}(\overline{\Omega})$
its adjoint $D^*$
preserves the adjoint boundary conditions (domain) of ${L}_\Omega^*$
and is continuous on the space
$C_{{L}^*}^{\infty}(\overline{\Omega})$, i.e.
that the operator
$D^*:C_{{L}^*}^{\infty}(\overline{\Omega})\to C_{{L}^*}^{\infty}(\overline{\Omega})$
is continuous.
Then we can extend $D$ to $\mathcal D'_{{L}}(\overline{\Omega})$ by
$$
\langle Dw,{\varphi}\rangle := \langle w, \overline{D^* \overline{\varphi}}\rangle \quad
(w\in \mathcal D'_{{L}}(\overline{\Omega}),\,  \varphi\in C_{{
L}^*}^{\infty}(\overline{\Omega})).
$$
This extends \eqref{EQ:duality-dist} from $L$ to other operators.
\end{remark}

\begin{definition}[Schwartz class on $\mathcal{I}$]
Let $\mathcal S(\mathcal{I})$ denote the space of rapidly decaying
functions $\varphi:\mathcal{I}\rightarrow\mathbb C$. That is,
$\varphi\in\mathcal S(\mathcal{I})$ if for any $M<\infty$ there
exists a constant $C_{\varphi, M}$ such that
$$
|\varphi(\xi)|\leq C_{\varphi, M}\langle\xi\rangle^{-M}
$$
holds for all $\xi\in\mathcal{I}$.

The topology on $\mathcal
S(\mathcal{I})$ is given by the seminorms $p_{k}$, where
$k\in\mathbb N_{0}$ and $$p_{k}(\varphi):=\sup_{\xi\in\mathcal{I}}\langle\xi\rangle^{k}|\varphi(\xi)|.$$
Continuous linear functionals on $\mathcal S(\mathcal{I})$ are of
the form
$$
\varphi\mapsto\langle u, \varphi\rangle:=\sum_{\xi\in\mathcal{I}}u(\xi)\varphi(\xi),
$$
where functions $u:\mathcal{I} \rightarrow \mathbb C$ grow at most
polynomially at infinity, i.e. there exist constants $M<\infty$
and $C_{u, M}$ such that
$$
|u(\xi)|\leq C_{u, M}\langle\xi\rangle^{M}
$$
holds for all $\xi\in\mathcal{I}$. Such distributions $u:\mathcal{I}
\rightarrow \mathbb C$ form the space of distributions which we denote by
$\mathcal S'(\mathcal{I})$.
\end{definition}
We now define the Fourier analysis on $C_{{L}}^{\infty}(\overline{\Omega})$ (and on $C_{{L}^*}^{\infty}(\overline{\Omega})$) introduced in  \cite{RT2016} using the eigenfunctions of $L$ and $L^*$.
\begin{definition} \label{FT}
We define the ${L}$-Fourier transform
$$
(\mathcal F_{{L}}f)(\xi)=(f\mapsto\widehat{f}):
C_{{L}}^{\infty}(\overline{\Omega})\rightarrow\mathcal S(\mathcal{I})
$$
by
\begin{equation}
\label{FourierTr}
\widehat{f}(\xi):=(\mathcal F_{{L}}f)(\xi)=\int\limits_{\Omega}f(x)\overline{v_{\xi}(x)}dx.
\end{equation}
Analogously, we define the ${L}^{\ast}$-Fourier
transform
$$
(\mathcal F_{{L}^{\ast}}f)(\xi)=(f\mapsto\widehat{f}_{\ast}):
C_{{L}^{\ast}}^{\infty}(\overline{\Omega})\rightarrow\mathcal
S(\mathcal{I})
$$
by
\begin{equation}\label{ConjFourierTr}
\widehat{f}_{\ast}(\xi):=(\mathcal F_{{
L}^{\ast}}f)(\xi)=\int\limits_{\Omega}f(x)\overline{u_{\xi}(x)}dx.
\end{equation}
\end{definition}

The expressions \eqref{FourierTr} and \eqref{ConjFourierTr}
are well-defined.
Moreover, we have

\begin{proposition}\label{LEM: FTinS}
The ${L}$-Fourier transform
$\mathcal F_{{L}}$ is a bijective homeomorphism from $C_{{
L}}^{\infty}(\overline{\Omega})$ to $\mathcal S(\mathcal{I})$.
Its inverse  $$\mathcal F_{{L}}^{-1}: \mathcal S(\mathcal{I})
\rightarrow C_{{L}}^{\infty}(\overline{\Omega})$$ is given by
\begin{equation}
\label{InvFourierTr} (\mathcal F^{-1}_{{
L}}h)(x)=\sum_{\xi\in\mathcal{I}}h(\xi)u_{\xi}(x),\quad h\in\mathcal S(\mathcal{I}),
\end{equation}
so that the Fourier inversion formula becomes
\begin{equation}
\label{InvFourierTr0}
f(x)=\sum_{\xi\in\mathcal{I}}\widehat{f}(\xi)u_{\xi}(x)
\quad \textrm{ for all } f\in C_{{
L}}^{\infty}(\overline{\Omega}).
\end{equation}
Similarly,  $\mathcal F_{{L}^{\ast}}:C_{{L}^{\ast}}^{\infty}(\overline{\Omega})\to \mathcal S(\mathcal{I})$
is a bijective homeomorphism and its inverse
$$\mathcal F_{{L}^{\ast}}^{-1}: \mathcal S(\mathcal{I})\rightarrow
C_{{L}^{\ast}}^{\infty}(\overline{\Omega})$$ is given by
\begin{equation}
\label{ConjInvFourierTr} (\mathcal F^{-1}_{{
L}^{\ast}}h)(x):=\sum_{\xi\in\mathcal{I}}h(\xi)v_{\xi}(x), \quad h\in\mathcal S(\mathcal{I}),
\end{equation}
so that the conjugate Fourier inversion formula becomes
\begin{equation}
\label{ConjInvFourierTr0} f(x)=\sum_{\xi\in\mathcal{I}}\widehat{f}_{\ast}(\xi)v_{\xi}(x)\quad \textrm{ for all } f\in C_{{
L^*}}^{\infty}(\overline{\Omega}).
\end{equation}
\end{proposition}
\begin{remark}[$L$-Schwartz Kernel theorem]
 Let us denote
$$ C^{\infty}_{{L}}(\overline{\Omega}\times \overline{\Omega}):=
C^{\infty}_{{L}}(\overline{\Omega})\otimes C^{\infty}_{{L}}(\overline{\Omega}),$$ 
and for the corresponding dual space let us write 
$\mathcal D'_{{L}}(\overline{\Omega}\times \overline{\Omega}):=
\left(C^{\infty}_{{L}}(\overline{\Omega}\times \overline{\Omega})\right)^\prime.$ 
It was proved in \cite{RT2016} that for a continuous linear operator   $$A:C^{\infty}_{{
L}}(\overline{\Omega})\rightarrow\mathcal D'_{{L}}(\overline{\Omega})$$
there exists a unique kernel $K\in \mathcal D'_{{L}}(\overline{\Omega}\times \overline{\Omega})$ such that
$$
\langle Af,g\rangle=\int\limits_{\Omega}\int\limits_{\Omega}K(x,y)f(x)g(y)dxdy,\,\,f,g\in C^{\infty}_{L}(\overline{\Omega}).
$$
Using the notion of the ${L}$-convolution in Section \ref{SEC:conv}, and  assuming WZ-condition,
the ${L}$-distribution $k_{A}\in\mathcal D'_{{L}}(\overline{\Omega}\times \overline{\Omega})$
defined by
\begin{equation} \label{EQ: KernelConv}
k_{A}(x,z):=k_A(x)(z):=\sum\limits_{\eta\in\mathcal{I}}u_{\eta}^{-1}(x)
\int\limits_{\Omega}K_{A}(x,y)u_{\eta}(y)dy \,
u_{\eta}(z),
\end{equation}
satisfies that
$$
Af(x)=(k_{A}(x)\sL f)(x).
$$
\end{remark}
\begin{proposition}\label{PROP:conv-kernel}
For a linear continuous operator $ A:C^{\infty}_{{
L}}(\overline{\Omega})\rightarrow \mathcal D'_{{L}}(\overline{\Omega})$ 
there exists a unique convolution kernel $k_{A}\in\mathcal D'_{{
L}}(\overline{\Omega}\times \overline{\Omega})$ such that
$$Af(x)=(f\sL k_{A}(x))(x),\quad f\in C^{\infty}_{{
L}}(\overline{\Omega}),$$ where we write $ k_{A}(x)(y):=k_{A}(x,y)$ 
in the sense of distributions. Also, for any linear continuous operator $ A:C^{\infty}_{{
L^{\ast}}}(\overline{\Omega})\rightarrow \mathcal D'_{{
L^{\ast}}}(\overline{\Omega})$  there exists a kernel $\widetilde{K}_{A}\in
\mathcal D'_{{L^{\ast}}}(\overline{\Omega}\times \overline{\Omega})$ such that for
all $f\in C^{\infty}_{{L^{\ast}}}(\overline{\Omega})$ we have
$$
Af(x)=\int\limits_{\Omega}\widetilde{K}_{A}(x,y)f(y)dy.
$$
If, in addition, $\{v_{\xi}: \,\,\, \xi\in\mathcal{I}\}$ is a ${\rm
WZ}$-system, then for a linear continuous operator
$A:C^{\infty}_{{L^{\ast}}}(\overline{\Omega})\rightarrow
\mathcal D'_{{L}^*}(\overline{\Omega})$ there exists the
convolution kernel $\widetilde{k}_{A}\in\mathcal D'_{{
L^{\ast}}}(\overline{\Omega}\times \overline{\Omega})$, such that
\begin{equation}\label{l*con}
    Af(x)=(f\sLs \widetilde{k}_{A}(x))(x),\quad f\in C^{\infty}_{{
L}^*}(\overline{\Omega}),
\end{equation} where we write $$\widetilde
k_{A}(x)(y):=\widetilde k_{A}(x,y)$$ in the sense of distributions.
\end{proposition}
In \eqref{l*con} we have used the ${L}^*$-convolution $\sLs,$ we refer to \eqref{EQ: CONV2} of Section \ref{SEC:conv} for the definition.

\subsection{${L}$-Quantization and and full symbols}
\label{SEC:quantization}

In this subsection we describe the ${L}$-quantization induced by the boundary value problem
${L}_\Omega$. 
\begin{definition}[${L}$-Symbols of operators on $\Omega$] \label{$L$--Symbols}
The ${L}$-symbol of a linear continuous
operator $$A:C^{\infty}_{{L}}(\overline{\Omega})\rightarrow
\mathcal D'_{{L}}(\overline{\Omega})$$ at $x\in\overline{\Omega}$ and
$\xi\in\mathcal{I}$ is defined by
$$\sigma_{A}(x, \xi):=\widehat{k_{A}(x)}(\xi)=\mathcal F_{{L}}(k_{A}(x))(\xi).$$
Hence, we can also write
$$\sigma_{A}(x, \xi)=\int\limits_{\Omega}k_{A}(x,y)\overline{v_{\xi}(y)}dy=
\langle k_{A}(x),\overline{v_{\xi}}\rangle.$$
\end{definition}

By the ${L}$-Fourier inversion formula the convolution kernel
can be regained from the symbol:
\begin{equation}
\label{Kernel} k_{A}(x, y)=\sum_{\xi\in\mathcal{I}}\sigma_{A}(x,
\xi)u_{\xi}(y),
\end{equation}
all in the sense of ${L}$-distributions. We now show that an
operator $A$ can be represented by its symbol:

\begin{theorem}[${L}$--quantization] \label{QuanOper}
Let $$A:C^{\infty}_{{L}}(\overline{\Omega})\rightarrow
C^{\infty}_{{L}}(\overline{\Omega})$$ be a continuous linear
operator with {L}-symbol $\sigma_{A}$. Then
\begin{equation}\label{Quantization}
Af(x)=\sum_{\xi\in\mathcal{I}}
u_{\xi}(x)\sigma_{A}(x, \xi) \widehat{f}(\xi)
\end{equation}
for every $f\in C^{\infty}_{{L}}(\overline{\Omega})$ and
$x\in\overline{\Omega}$.
The {L}-symbol $\sigma_{A}$ satisfies
\begin{equation}\label{FormSymb}
\sigma_{A}(x,\xi)=u_{\xi}(x)^{-1}(Au_{\xi})(x)
\end{equation}
for all $x\in\overline{\Omega}$ and $\xi\in\mathcal{I}$.
\end{theorem}
In a similar fashion, the $L^*$-symbols of operators on $\Omega$ and $L^*$-quantization can be introduced. We refer to \cite{RT2016} for more detail.

\subsection{Difference operators and symbolic calculus}
\label{SEC:differences}
In this subsection we discuss difference operators that will be useful to define symbol
classes for the symbolic calculus of operators.

Let $q_{j}\in C^{\infty}({\Omega}\times{\Omega})$, $j=1,\ldots,l$, be a given family
of smooth functions.
We will call the collection of $q_j$'s { {$L$}-strongly admissible} if the following properties hold:
\begin{itemize}
\item For every $x\in\overline{\Omega}$, the multiplication by $q_{j}(x,\cdot)$
is a continuous linear mapping on
 $C^{\infty}_{{L}}(\overline{\Omega})$, for all $j=1,\ldots,l$;
\item $q_{j}(x,x)=0$ for all $j=1,\ldots,l$;
\item $
{\rm rank}(\nabla_{y}q_{1}(x,y), \ldots, \nabla_{y}q_{l}(x,y))|_{y=x}=n;
$
\item  the diagonal in $\overline{\Omega}\times \overline{\Omega}$ is the only set when all of
$q_j$'s vanish:
$$
\bigcap_{j=1}^l \left\{(x,y)\in\overline{\Omega}\times \overline{\Omega}: \, q_j(x,y)=0\right\}=\{(x,x):\, x\in\overline{\Omega}\}.
$$
\end{itemize}
We will use the multi-index notation
$$
q^{\alpha}(x,y):=q^{\alpha_1}_{1}(x,y)\cdots q^{\alpha_l}_{l}(x,y).
$$
Analogously, the notion of an ${L}^{*}$-strongly admissible collection suitable for the
conjugate problem can be introduced. 

Now,we record the Taylor expansion formula with respect to a family of $q_j$'s,
which follows from expansions of functions $g$ and
$q^{\alpha}(e,\cdot)$ by the common Taylor series:

\begin{proposition}\label{TaylorExp}
Any smooth function $g\in C^{\infty}({\Omega})$ can be
approximated by Taylor polynomial type expansions, i.e. for $e\in\Omega$, we have
$$g(x)=\sum_{|\alpha|<
N}\frac{1}{\alpha!}D^{(\alpha)}_{x}g(x)|_{x=e}\, q^{\alpha}(e,x)+\sum_{|\alpha|=
N}\frac{1}{\alpha!}q^{\alpha}(e,x)g_{N}(x)
$$
\begin{equation}
\sim\sum_{\alpha\geq
0}\frac{1}{\alpha!}D^{(\alpha)}_{x}g(x)|_{x=e}\, q^{\alpha}(e,x)
\label{TaylorExpFormula}
\end{equation}
in a neighborhood of $e\in\Omega$, where $g_{N}\in
C^{\infty}({\Omega})$ and
$D^{(\alpha)}_{x}g(x)|_{x=e}$ can be found from the recurrent formulae:
$D^{(0,\cdots,0)}_{x}:=I$ and for $\alpha\in\mathbb N_0^l$,
$$
\mathsf
\partial^{\beta}_{x}g(x)|_{x=e}=\sum_{|\alpha|\leq|\beta|}\frac{1}{\alpha!}
\left[\mathsf
\partial^{\beta}_{x}q^{\alpha}(e,x)\right]\Big|_{x=e}D^{(\alpha)}_{x}g(x)|_{x=e},
$$
where $\beta=(\beta_1, \ldots, \beta_n)$ and
$
\partial^{\beta}_{x}=\frac{\partial^{\beta_{1}}}{\partial x_{1}^{\beta_{1}}}\cdots
\frac{\partial^{\beta_{n}}}{\partial x_{n}^{\beta_{n}}}.
$
\end{proposition}

Analogously, any function $g\in C^{\infty}({\Omega})$
can be approximated by Taylor polynomial type expansions corresponding to
the adjoint problem using the $L^*$-strongly admissible collection.

It can be seen that operators $D^{(\alpha)}$ 
is differential operator of order
$|\alpha|$.

\begin{definition}\label{DEF: DifferenceOper}\label{DEF: DifferenceOper_2}
For WZ-systems, we define difference operator $\Delta_{q,(x)}^{\alpha}$ acting
on Fourier coefficients by any of the following equal expressions
\begin{align*}
\Delta_{q,(x)}^{\alpha}\widehat{f}(\xi)& = u_{\xi}^{-1}(x)
\int\limits_{\Omega}\Big[\int\limits_{\Omega}q^{\alpha}(x,y)F(x,y,z)f(z)dz\Big]u_{\xi}(y)dy
\\
& = u_{\xi}^{-1}(x)
\sum_{\eta\in\mathcal{I}}\mathcal
F_{L}\Big(q^{\alpha}(x,\cdot)u_{\xi}(\cdot)\Big)(\eta)\widehat{f}(\eta)u_{\eta}(x)
\\
& = u_{\xi}^{-1}(x) \left([q^{\alpha}(x,\cdot)u_{\xi}(\cdot)]\sL
f\right)(x).
\end{align*}
Similarly, we can  define the action of the difference operator
$\widetilde{\Delta}_{q,(x)}^{\alpha}$ acting on adjoint Fourier
coefficients. 
\end{definition}
For simplicity, if there is no confusion, for a fixed collection
of $q_j$'s, instead of $\Delta_{q,(x)}$ and
$\widetilde{\Delta}_{\widetilde{q},(x)}$ we will often simply write
$\Delta_{(x)}$ and $\widetilde{\Delta}_{(x)}$.

\medskip
Using such difference operators and derivatives $D^{(\alpha)}$ from
Proposition \ref{TaylorExp}
we can now define classes of symbols.

\begin{definition}[Symbol class $S^m_{\rho,\delta}(\overline{\Omega}\times\mathcal{I})$]\label{DEF: SymClass}
Let $m\in\mathbb R$ and $0\leq\delta,\rho\leq 1$. The ${L}$-symbol class
$S^m_{\rho,\delta}(\overline{\Omega}\times\mathcal{I})$ consists of
those functions $a(x,\xi)$ which are smooth in $x$ for all
$\xi\in\mathcal{I}$, and which satisfy
\begin{equation}\label{EQ:symbol-class}
  \left|\Delta_{(x)}^\alpha D^{(\beta)}_{x} a(x,\xi) \right|
        \leq C_{a\alpha\beta m}
                \ \langle\xi\rangle^{m-\rho|\alpha|+\delta|\beta|}
\end{equation}
for all $x\in\overline{\Omega}$, for all $\alpha,\beta\geq 0$, and for all $\xi\in\mathcal{I}$.
Here the operators $D^{(\beta)}_{x}$ are defined in Proposition
\ref{TaylorExp}. We will often denote them simply by $D^{(\beta)}$.
In \eqref{EQ:symbol-class}, we assume that the inequality is satisfied for $x\in\overline{\Omega}$ and
it extends to the closure $\overline\Omega$.
Furthermore, we define
$$
S^{\infty}_{\rho,\delta}(\overline{\Omega}\times\mathcal{I}):=\bigcup\limits_{m\in\mathbb
R}S^{m}_{\rho,\delta}(\overline{\Omega}\times\mathcal{I})
$$
and
$$
S^{-\infty}(\overline{\Omega}\times\mathcal{I}):=\bigcap\limits_{m\in\mathbb
R}S^{m}(\overline{\Omega}\times\mathcal{I}).
$$
When we have two $L$-strongly admissible collections, expressing one in terms of
the other similarly to Proposition \ref{TaylorExp} and arguing similarly to
\cite{Ruzhansky-Turunen-Wirth:JFAA}, we can convince ourselves that for $\rho>\delta$ the
definition of the symbol class does not depend on the choice of an
 $L$-strongly admissible collection.
 
A symbol $\sigma_A\in S^m_{\rho,\delta}(\overline{\Omega}\times\ind)$ is said to be 
$L$-{\it elliptic} if there exist constants $C_0>0$ and
$N_0\in\mathbb N$ such that
\begin{equation}\label{elliptic}
  |\sigma_A(x,\xi)|
  \geq C_0 \langle\xi\rangle^m
\end{equation}
for all $(x,\xi)\in\overline{\Omega}\times\ind$ for which
$\langle \xi \rangle \geq N_0.$

\end{definition}

If $a\in S^m_{\rho,\delta}(\overline{\Omega}\times\mathcal{I})$, it is convenient to
denote by $a(X,D)= {\rm Op}_L(a)$  the corresponding ${
L}$-pseudo-differential operator defined by
\begin{equation}\label{EQ: L-tor-pseudo-def}
  {\rm Op}_L(a)f(x)=a(X,D)f(x):=\sum_{\xi\in\mathcal{I}} u_{\xi}(x)\ a(x,\xi)\widehat{f}(\xi).
\end{equation}
The set of operators ${\rm Op}_L(a)$ of the form
(\ref{EQ: L-tor-pseudo-def}) with $a\in
S^m_{\rho,\delta}(\overline{\Omega}\times\mathcal{I})$ will be denoted by
${\rm Op}_L(S^m_{\rho,\delta} (\overline{\Omega}\times\mathcal{I}))$, or by
$\Psi^m_{\rho,\delta} (\overline{\Omega}\times\mathcal{I})$. If an
operator $A$ satisfies $A\in{\rm
Op_L}(S^m_{\rho,\delta}(\overline{\Omega}\times\mathcal{I}))$, we denote
its ${L}$-symbol by $\sigma_{A}=\sigma_{A}(x, \xi), \,\,
x\in\overline{\Omega}, \, \xi\in\mathcal{I}$. Naturally,
$\sigma_{a(X,D)}(x,\xi)=a(x,\xi)$.

Analogously, we define the ${L^{\ast}}$-symbol class
$\widetilde{S}^m_{\rho,\delta}(\overline{\Omega}\times\mathcal{I})$  and 
the corresponding ${
L^{\ast}}$-pseudo-differential operator ${\rm Op}_{L^*}(a)$  can be defined.

\begin{remark}\label{REM: Topology of SymClass}
{\rm (Topology on $S^{m}_{\rho,
\delta}(\overline{\Omega}\times\mathcal{I})$ ($\widetilde{S}^{m}_{\rho,
\delta}(\overline{\Omega}\times\mathcal{I})$)).} The set $S^{m}_{\rho,
\delta}(\overline{\Omega}\times\mathcal{I})$ ($\widetilde{S}^{m}_{\rho,
\delta}(\overline{\Omega}\times\mathcal{I})$) of symbols has a natural
topology. Let us consider the functions $p_{\alpha\beta}^{l}:
S^{m}_{\rho,
\delta}(\overline{\Omega}\times\mathcal{I})\rightarrow\mathbb R$
($\widetilde{p}_{\alpha\beta}^{l}: \widetilde{S}^{m}_{\rho,
\delta}(\overline{\Omega}\times\mathcal{I})\rightarrow\mathbb R$) defined
by
$$
p_{\alpha\beta}^{l}(\sigma):={\rm
sup}\left[\frac{\left|\Delta_{(x)}^{\alpha}D^{(\beta)}\sigma(x,
\xi)\right|}{\langle\xi\rangle^{l-\rho|\alpha|+\delta|\beta|}}:\,\,
(x, \xi)\in\overline{\Omega}\times\mathcal{I}\right]
$$
$$
\left(\widetilde{p}_{\alpha\beta}^{l}(\sigma):={\rm
sup}\left[\frac{\left|\widetilde{\Delta}_{(x)}^{\alpha}\widetilde{D}^{(\beta)}\sigma(x,
\xi)\right|}{\langle\xi\rangle^{l-\rho|\alpha|+\delta|\beta|}}:\,\,
(x, \xi)\in\overline{\Omega}\times\mathcal{I}\right]\right).
$$
Now $\{p_{\alpha\beta}^{l}\}$
($\{\widetilde{p}_{\alpha\beta}^{l}\}$) is a countable family of seminorms,
and they define a
Fr\'echet topology on $S^{m}_{\rho,
\delta}(\overline{\Omega}\times\mathcal{I})$
($\widetilde{S}^{m}_{\rho, \delta}(\overline{\Omega}\times\mathcal{I})$). Due to the bijective correspondence of ${\rm
Op}_L(S^{m}_{\rho, \delta}(\overline{\Omega}\times\mathcal{I}))$ and
$S^{m}_{\rho, \delta}(\overline{\Omega}\times\mathcal{I})$ (${\rm
Op}_{L^*}(\widetilde{S}^{m}_{\rho,
\delta}(\overline{\Omega}\times\mathcal{I}))$ and
$\widetilde{S}^{m}_{\rho, \delta}(\overline{\Omega}\times\mathcal
I)$), this directly topologises also the set of operators. These spaces
are not normable, and the topologies have but a marginal role.
\end{remark}

The next theorem is a prelude to asymptotic expansions, which are
the main tool in the symbolic analysis of ${
L}$-pseudo-differential operators.

\begin{theorem}[Asymptotic sums of symbols] Let $(m_{j})_{j=0}^{\infty}\subset\mathbb
R$ be a sequence such that $m_{j}>m_{j+1}$, and
$m_{j}\rightarrow-\infty$ as $j\rightarrow\infty$, and
$\sigma_{j}\in
S^{m_{j}}_{\rho,\delta}(\overline{\Omega}\times\mathcal{I})$ for all
$j\in\mathcal{I}$. Then there exists an ${L}$-symbol $\sigma\in
S^{m_{0}}_{\rho,\delta}(\overline{\Omega}\times\mathcal{I})$ such that
for all $N\in\mathcal{I}$,
$$
\sigma\stackrel{m_{N},\rho,\delta}{\sim}\sum_{j=0}^{N-1}\sigma_{j}.
$$
\end{theorem}

Next theorem shows that the non-harmonic calculus is closed under taking adjoints \cite{RT2016}.

\begin{theorem}[Adjoint operators]
Let $0\leq\delta<\rho\leq 1$. Let $A\in {\rm Op}_L
(S^m_{\rho,\delta}(\overline{\Omega}\times\mathcal{I}))$.
Assume that the conjugate symbol
class $\widetilde{S}^{m}_{\rho,\delta}(\overline{\Omega}\times\mathcal{I})$
is defined with strongly admissible
functions $\widetilde{q}_{j}(x,y):=\overline{q_{j}(x,y)}$ which are ${L}^{*}$-strongly admissible.
Then the adjoint of $A$ satisfies
$A^{\ast}\in {\rm Op_{L^*}}(\widetilde{S}^{m}_{\rho,\delta}(\overline{\Omega}\times\mathcal{I}))$,
with its ${L}^*$-symbol
$\tau_{A^*}\in \widetilde{S}^{m}_{\rho,\delta}(\overline{\Omega}\times\mathcal{I})$
having the asymptotic expansion
$$
\tau_{A^*}(x,\xi) \sim \sum_\alpha \frac{1}{\alpha!}
\widetilde \Delta_x^\alpha D_x^{(\alpha)}\overline{\sigma_A(x,\xi)}.
$$
\end{theorem}

We now formulate the composition formula \cite{RT2016}.
\begin{theorem}\label{Composition}
Let $m_{1}, m_{2}\in\mathbb R$ and $\rho>\delta\geq0$. Let $A,
B:C_{{L}}^{\infty}(\overline{\Omega})\rightarrow C_{{
L}}^{\infty}(\overline{\Omega})$ be continuous and linear, and assume that
their {L}-symbols satisfy
\begin{align*}
|\Delta_{(x)}^{\alpha}\sigma_{A}(x,\xi)|&\leq
C_{\alpha}\langle\xi\rangle^{m_{1}-\rho|\alpha|},\\
|D^{(\beta)}\sigma_{B}(x,\xi)|&\leq
C_{\beta}\langle\xi\rangle^{m_{2}+\delta|\beta|},
\end{align*}
for all $\alpha,\beta\geq 0$, uniformly in $x\in\overline{\Omega}$ and
$\xi\in\mathcal{I}$.
Then
\begin{equation}
\sigma_{AB}(x,\xi)\sim\sum_{\alpha\geq 0}
\frac{1}{\alpha!}(\Delta_{(x)}^{\alpha}\sigma_{A}(x,\xi))D^{(\alpha)}\sigma_{B}(x,\xi),
\label{CompositionForm}
\end{equation}
where the asymptotic expansion means that for every $N\in\mathbb N$ we have
$$
|\sigma_{AB}(x,\xi)-\sum_{|\alpha|<N}\frac{1}{\alpha!}(\Delta_{(x)}^{\alpha}\sigma_{A}(x,\xi))D^{(\alpha)}\sigma_{B}(x,\xi)|\leq
C_{N}\langle\xi\rangle^{m_{1}+m_{2}-(\rho-\delta)N}.
$$
\end{theorem}

For the proof of these important properties of the non-harmonic pseudo-differential calculus we refer the reader to  \cite{RT2016}. In the next section we analyse the Dixmier traceability and the expansion of traces of $L$-elliptic operators. They are operators whose symbols satisfy \eqref{Iesparametrix}. 

\begin{theorem}\label{IesTParametrix} Let $m\in \mathbb{R},$ and let $0\leqslant \delta<\rho\leqslant 1.$  Let  $a=a(x,\xi)\in {S}^{m}_{\rho,\delta}(M\times \mathcal{I}).$  Assume also that $a(x,\xi)$ is invertible for every $(x,\xi)\in M\times\mathcal{I}
,$ and satisfies
\begin{equation}\label{Iesparametrix}
 \Vert  \langle \cdot\rangle^m a^{-1} \Vert_{L^{\infty}(\mathcal{I})}:= \sup_{(x,\xi)\in M\times\mathcal{I}} \vert \langle \xi\rangle^ma(x,\xi)^{-1}\vert<\infty.
\end{equation}Then, there exists $B\in {S}^{-m}_{\rho,\delta}(M\times \mathcal{I}),$ such that $AB-I,BA-I\in {S}^{-\infty}(M\times \mathcal{I}). $ Moreover, the symbol of $B$ satisfies the following asymptotic expansion
\begin{equation}\label{AE}
    \widehat{B}(x,\xi)\sim \sum_{N=0}^\infty\widehat{B}_{N}(x,\xi),\,\,\,(x,\xi)\in M\times \mathcal{I},
\end{equation}where $\widehat{B}_{N}(x,\xi)\in {S}^{-m-(\rho-\delta)N}_{\rho,\delta}(M\times \mathcal{I})$ obeys to the inductive  formula
\begin{equation}\label{conditionelip}
    \widehat{B}_{N}(x,\xi)=-a(x,\xi)^{-1}\left(\sum_{k=0}^{N-1}\sum_{|\gamma|=N-k}(\Delta_{(x)}^{\gamma} a(x,\xi))(D_{x}^{(\gamma)}\widehat{B}_{k}(x,\xi))\right),\,\,N\geqslant 1,
\end{equation}with $ \widehat{B}_{0}(x,\xi)=a(x,\xi)^{-1}.$
\end{theorem}

We have the following Calder\'on-Vaillancourt Theorem from \cite{CardonaKumarRuzhanskyTokmagambetov2020II}. For the $L^{p}$-$L^q$-boundedness of pseudo-differential operators in the setting of non-harmonic analysis we refer the reader to \cite{CardVishTokRuzI}.
\begin{theorem}
\label{L2}  
Let $a(x,D):C^\infty_L(M)\rightarrow\mathcal{D}'_L(M)$ be a pseudo-differential  operator with symbol  $a\in {S}^{0}_{\rho,\delta}( M\times \mathcal{I})$. Then $a(x,D)$ extends to a bounded operator on $L^2({M})$. 
\end{theorem}
Throughout the paper, we shall use the notation $A \lesssim B$ to indicate $A\leq cB $ for a suitable constant $c >0,$   whereas $A \asymp B$ if $A\leq cB$ and $B\leq d A$, for suitable $c, d >0.$

\section{Asymptotic expansions for  regularised traces of  $L$-elliptic  global pseudo-differential operators}\label{expansionshere}
In this section we will study the trace for the heat semigroup associated with $L$-elliptic positive left-invariant operators and also  regularised traces of $L$-elliptic operators. In this section we do not require the WZ-condition on eigenfunctions.  We make the following standing hypothesis for the Weyl eigenvalue counting function $N(\lambda)$  of the operator $(1+LL^\circ)^{\frac{1}{2\nu}}$ for the rest of the paper.
\begin{itemize}
    \item{(WL):} The operator $(1+L^\circ L)^{\frac{1}{2\nu}}$ satisfies the Weyl law for some $Q>0.$ This means that $N(\lambda):=|\{\xi\in \mathcal{I}:(1+|\lambda_\xi|^2)^\frac{1}{2\nu}\leq \lambda \}|\sim \lambda^{Q}.$ Certainly, we consider the smallest $Q$ with this property when $\lambda\rightarrow\infty.$ 
\end{itemize}
\begin{remark}
By replacing $\lambda>0,$ by $\lambda^{\frac{1}{\nu}},$ in (WL) we have $$N(\lambda^{\frac{1}{\nu}}):=|\{\xi\in \mathcal{I}:(1+|\lambda_\xi|^2)^\frac{1}{2}\leq \lambda \}|\sim \lambda^{\frac{Q}{\nu}}.$$
\end{remark}
\begin{remark} The class of
symbols $\sigma(\xi),$ $\xi\in \mathcal{I},$ in $ {S}^{m}_{\rho,\delta}(M\times   \mathcal{I})$ independent of the first argument $x\in M,$ will be denoted by $ {S}^{m}_{\rho}(  \mathcal{I}).$ The associated operators are called  $L$-Fourier multipliers.
\end{remark}
The following lemma justifies the fact that (WL) is natural to assume when dealing with global H\"ormander classes related to the nonharmonic analysis on manifolds. 

\begin{lemma}\label{equivalence} Assumption $(WL)$ above is equivalent to Assumption \ref{Assumption_4}. 
\end{lemma}
\begin{proof}
    Suppose that Assumption \ref{Assumption_4} holds. Then, there is
$s_0\in\mathbb R$  such that we have
$$\sum_{\xi\in\mathcal{I}} \langle\xi\rangle^{-s_0}<\infty.$$ Observe that,
\begin{align*}
    \sum_{\xi\in\mathcal{I}}\langle\xi\rangle^{-s_0}&=\sum_{k=0}^\infty\sum_{\xi: 2^{k}\leq \langle\xi\rangle<2^{k+1}}\langle\xi\rangle^{-s_0}\asymp \sum_{k=0}^\infty\sum_{\xi: 2^{k}\leq \langle\xi\rangle<2^{k+1}}2^{-ks_0}\\
    &\lesssim \sum_{k=0}^\infty2^{-ks_0}N(2^{k})<\infty.
\end{align*}
From the convergence of the last series we deduce the estimate $2^{-ks_0}N(2^{k})=O(\frac{1}{k})$ which implies $N(\lambda)=O(\frac{1}{\log(\lambda)}\lambda^{s_0})$ when $\lambda\rightarrow\infty$ for some $s_0.$ So, we have in particular $N(\lambda)=O(\lambda^{s_0}),$ when $\lambda\rightarrow\infty. $ This proves the first part of the lemma. Now, let us assume that $N(\lambda)=O(\lambda^{Q})$ for some $Q>0.$ Then, for $s>0,$ 
\begin{align*}
    \sum_{\xi\in\mathcal{I}}\langle\xi\rangle^{-s}&=\sum_{k=0}^\infty\sum_{\xi: 2^{k}\leq \langle\xi\rangle<2^{k+1}}\langle\xi\rangle^{-s}\asymp \sum_{k=0}^\infty\sum_{\xi: 2^{k}\leq \langle\xi\rangle<2^{k+1}}2^{-ks}\\
    &\lesssim \sum_{k=0}^\infty2^{-ks+kQ}<\infty,
\end{align*}if and only if $s>Q.$ So, if we take $s_0>Q,$ Assumption \ref{Assumption_4} is satisfied. The proof is complete.
\end{proof}

So, we start with the following Pleijel type formula. We record that an operator $A:C^\infty_L(M)\rightarrow\mathcal{D}'_L(M)$ is $L$-elliptic if its symbol satisfies \eqref{Iesparametrix}.
 \begin{lemma}\label{asymptotictracemultiplier}  
Let $M=\overline{\Omega}$ be a smooth manifold with (possibly empty)  boundary $\partial \Omega.$   For  $0\leqslant \rho\leqslant 1,$ let us consider a positive   $L$-elliptic continuous linear operator $A:C^\infty_L(M)\rightarrow\mathcal{D}'_L(M)$ with symbol  $\sigma\in {S}^{m}_{\rho}(  \mathcal{I})$, $m> 0$. Then the heat trace of $A$ has an asymptotic behaviour of the form
\begin{equation}\label{asymp1}
   \textnormal{\bf {Tr}} (e^{-tA})\sim c_{m,Q} t^{-\frac{Q}{m}}\times \int\limits_{t^{\frac{1}{m}}}^{\infty}e^{-s^m}s^{Q} ds,\,\,\forall t>0.
\end{equation} Moreover, we have the following asymptotic expansion
\begin{equation}
 \textnormal{\bf {Tr}}(e^{-sA})=s^{-\frac{Q}{m}}\left(  \sum_{k=0}^{\infty}a_ks^{\frac{k}{m}}\right),\,\,s\rightarrow 0^{+}  .
\end{equation}
\end{lemma}
 \begin{proof}Note that $A$ is densely defined and positive on $L^2(M),$ so it admits a self-adjoint extension.
The symbol of $A$ is determined by the sequence $ 
    \sigma\equiv[\sigma(\xi)]_{\xi\in \mathcal{I}}.$ 
Using the spectral mapping theorem we have
\begin{equation*}
    \textnormal{spectrum}(e^{-tA})=\{e^{-t\sigma(\xi)}:\xi\in  \mathcal{I}\}.
    \end{equation*} 
So, we have
    \begin{align*}
        \textnormal{\bf {Tr}} (e^{-tA})=\sum_{\xi\in   \mathcal{I}}e^{-t\sigma(\xi)}.
    \end{align*} 
Using the fact that $A$ is $L$-elliptic, one obtains
    \begin{equation*}
        \sigma(\xi)^{-1}(1+|\lambda_{\xi}|^2)^{\frac{m}{2\nu}}\leqslant \Vert  \langle \cdot\rangle^m \sigma^{-1} \Vert_{L^{\infty}(\mathcal{I})}:=  \sup_{\xi\in  \mathcal{I}}\vert \sigma(\xi)^{-1}\langle \xi \rangle^m\vert<\infty.
    \end{equation*} 
Consequently,
    \begin{equation*}
        \sigma(\xi)(1+|\lambda_{\xi}|^2)^{-\frac{m}{2\nu }}\geqslant    \Vert  \langle \cdot\rangle^m \sigma^{-1} \Vert_{L^{\infty}(\mathcal{I})}^{-1}.
    \end{equation*}
Now, in view of \eqref{EQ:symbol-class}, observe that from the hypohesis $\sigma\in {S}^{m}_{\rho}(  \mathcal{I})$ we have
\begin{equation*}
  \left|\sigma(\xi) \right|
        \leq C_{\sigma, m}
                 \langle\xi\rangle^{m},\,\,\xi\in \mathcal{I}.
\end{equation*}
Consequently,
\begin{equation*}
  \sigma(\xi)(1+|\lambda_{\xi}|^2)^{-\frac{m}{2\nu}}\leqslant   \Vert  \langle \cdot\rangle^{-m} \sigma \Vert_{L^{\infty}(\mathcal{I})}:=\sup_{\xi\in  \mathcal{I}}\vert \sigma(\xi)\langle \xi \rangle^{-m}\vert\leq  C_{\sigma, m}.   
\end{equation*}These inequalities reduce the problem of estimating the trace  $\textnormal{\bf {Tr}} (e^{-tA})$ to the problem of computing $ \textnormal{\bf {Tr}}(e^{-t(1+L^{\circ}L)^{\frac{m}{2\nu}}}),$ because they imply that
\begin{equation}\label{ellipticityofE}
  \sigma(\xi)(1+|\lambda_{\xi}|^2)^{-\frac{m}{2\nu}},\,\sigma(\xi)^{-1}(1+|\lambda_{\xi}|^2)^{\frac{m}{2\nu}}  =o(1).
\end{equation}
Indeed,
\begin{align*}
     & \textnormal{\bf {Tr}}(e^{-tA})=\sum_{\xi\in   \mathcal{I}}e^{-t\sigma(\xi)}=\sum_{\xi\in   \mathcal{I}}e^{-t\sigma(\xi) (1+|\lambda_{\xi}|^2)^{-\frac{m}{2\nu}}(1+|\lambda_{\xi}|^2)^{\frac{m}{2\nu}}      }\\
     &\asymp \sum_{\xi\in   \mathcal{I}}e^{-t(1+|\lambda_{\xi}|^2)^{\frac{m}{2\nu}}      }= \textnormal{\bf {Tr}}(e^{-t(1+L^{\circ}L)^{\frac{m}{2\nu}}}).
\end{align*}Now, we will use the Weyl law for $(1+L^\circ L)^{\frac{1}{2\nu}}$. Observe that,
\begin{align*}
   \textnormal{\bf {Tr}}(e^{-t(1+L^{\circ}L)^{\frac{m}{2\nu}}})=\sum_{k=0}^{\infty}\sum_{\xi:2^{k}\leqslant (1+|\lambda_{\xi}|^2)^\frac{1}{2\nu} <2^{k+1}} e^{-t(1+|\lambda_{\xi}|^2)^{\frac{m}{2\nu}}      }.
\end{align*}Because,
\begin{align*}
    \sum_{\xi:2^{k}\leqslant (1+|\lambda_{\xi}|^2)^\frac{1}{2\nu} <2^{k+1}} e^{-t(1+|\lambda_{\xi}|^2)^{\frac{m}{2\nu}}      }\asymp \sum_{\xi:2^{k}\leqslant (1+|\lambda_{\xi}|^2)^\frac{1}{2\nu } <2^{k+1}} {}e^{-t2^{km}}, 
\end{align*} 
we have 
\begin{align*}
     \textnormal{\bf {Tr}}(e^{-t(1+L^{\circ}L)^{\frac{m}{2\nu}}})&=\sum_{k=0}^{\infty} e^{-t2^{km}}\sum_{\xi:2^{k}\leqslant (1+|\lambda_{\xi}|^2)^\frac{1}{2\nu} <2^{k+1}} 1 \\
    &=\sum_{k=0}^{\infty} e^{-t2^{km}}N(2^k)=\sum_{k=0}^{\infty} e^{-t2^{km}}2^{kQ}\\
    &=\sum_{k=0}^{\infty} e^{-t2^{km}}2^{k(Q-1)}2^{k}.
\end{align*} 
Observe that
\begin{align*}
    \sum_{k=0}^{\infty} e^{-t2^{km}}2^{k(Q-1)}2^{k}\asymp \int\limits_{1}^{\infty}e^{-t\lambda^m}\lambda^{Q-1}d\lambda=t^{-\frac{Q}{m}}\int\limits_{t^{\frac{1}{m}}}^{\infty}e^{-s^m}s^{Q-1}ds.
\end{align*}The condition $m>0,$ implies that $g(t):=\int\limits_{t}^{\infty}e^{-s^m}s^{Q-1}ds,$ is smooth and real-analytic on $\mathbb{R}^{+}:=(0,\infty),$ admitting a Taylor expansion of the form 
\begin{align*}
  g(s)=  \sum_{k=0}^{\infty}a_k's^k\,\,\,s\rightarrow 0^+.
\end{align*}
So,  we have the estimate $ \textnormal{\bf {Tr}} (e^{-sA})\sim c_{m,Q}s^{-\frac{Q}{m}}g(s),$ for some positive constant $c_{m,Q}.$ On the other hand, we deduce that $F(s):=s^{\frac{Q}{m }}\textnormal{\bf {Tr}} (e^{-sA})$ is a real-analytic function and its Taylor expansion at $s=0,$ has the form: $\sum_{k=0}^{\infty}a_k's^{\frac{k}{m}},$ which implies the following expansion,
\begin{equation*}
 \textnormal{\bf {Tr}}(e^{-sA})=s^{-\frac{Q}{m}}\left(  \sum_{k=0}^{\infty}a_ks^{\frac{k}{m}}\right),\,\,s\rightarrow 0^{+}  .
\end{equation*}
Thus, we end the proof. \end{proof}
 
\begin{remark}Observe that under the conditions of Theorem \ref{asymptotictracemultiplier}, we have
\begin{equation}
    \textnormal{\bf {Tr}} (e^{-tA})\sim c_{m,Q, t} t^{-\frac{Q}{m}},\,\,\forall t>0,
\end{equation} where $c_{m,Q, t}:=c_{m,Q} \int\limits_{t^{\frac{1}{m}} }^{\infty}e^{-s^m}s^{Q}ds.$ For $t\rightarrow\infty,$ $c_{m,Q, t}\rightarrow 0^{+},$ and in general,
\begin{equation*}
    0<c_{m,Q, t} \lesssim \int\limits_{0}^{\infty}e^{-s^m}s^{Q}ds=o( 1),\,\,\,0\leqslant t<\infty.
\end{equation*}So, \eqref{asymp1} implies the following estimate
\begin{equation*}
     \textnormal{\bf {Tr}} (e^{-tA})\sim c_{m,Q} t^{-\frac{Q}{m}}.
\end{equation*}
\end{remark} 
Now, we study other kinds of singularities appearing in traces of the form $  \textnormal{\bf {Tr}}(Ae^{-t(1+L^{\circ}L)^{ \frac{q}{2\nu} }  }).$  
 
\begin{theorem}\label{asymptotictracemultiplier2}
Let $M=\overline{\Omega}$ be a smooth manifold with (possibly empty)  boundary $\partial \Omega.$ For $0\leqslant \rho, \delta \leqslant 1,$ let us consider a continuous linear operator $A:C^\infty_L(M)\rightarrow\mathcal{D}'_L(M)$ with symbol  $\sigma\in {S}^{m}_{\rho,\delta}(M\times   \mathcal{I})$. Let $\mathcal{M}_q:=(1+L^{\circ}L)^{\frac{q}{2\nu}}$ with $q>0$. Then, 
\begin{equation}\label{asymp122}
  |\textnormal{\bf {Tr}}   (Ae^{-t\mathcal{M}_q})|\leq  c_{m,Q} t^{-\frac{Q+m}{q}}\int\limits_{t^{\frac{1}{q}}}^{\infty}e^{-s^{q}}s^{Q+m-1}ds,\,\,\forall t>0.
\end{equation}In particular, for $m=-Q,$ we have
\begin{equation}\label{asymp1212}
   |\textnormal{\bf {Tr}}  (Ae^{-t\mathcal{M}_q})|\leq  -c_Q\frac{1}{q}\log(t),\,\,\forall t\in (0,1),
\end{equation}while for $m>-Q,$ we have the estimate
\begin{equation}\label{asymp1212'}
 |\textnormal{\bf {Tr}}(Ae^{-t\mathcal{M}_q})|\leq C_{m,Q} t^{-\frac{Q+m}{q}}\left(  \sum_{k=0}^{\infty}b_k't^{\frac{k}{q}}\right),\,\,t\rightarrow 0^{+}.
\end{equation}
\end{theorem}
\begin{proof}
We will follow the same approach as in Theorem \ref{asymptotictracemultiplier}. Because the trace of $Ae^{-t\mathcal{M}_q}$ is the integral of its Schwartz kernel over the diagonal (see \cite{DRT2017} for details), we have

 \begin{align*}
    \textnormal{\bf {Tr}}  (Ae^{-t\mathcal{M}_q})&=\int\limits_{M}\sum_{\xi\in   \mathcal{I}} [\sigma(x,\xi)e^{-t\langle \xi \rangle^{q}}]\, u_\xi(x)\, \overline{v_\xi(x)} dx\\
     &=\int\limits_{M}\sum_{\xi\in   \mathcal{I}} [\sigma(x,\xi)  \langle \xi \rangle^{-m}\langle \xi \rangle^{m} e^{-t\langle \xi \rangle^{q}}]\,u_\xi(x)\, \overline{v_\xi(x)}dx.
\end{align*}Let us denote
\begin{equation}
    \sigma(x,\xi)  \langle \xi \rangle^{-m}=:\lambda(x,\xi),\,\,\langle \xi \rangle^{m} e^{-t\langle \xi \rangle^{q}}=:\Omega_{t}(\xi).
\end{equation}
Now, we can write
\begin{align} \label{lackofselfad}
     \textnormal{\bf {Tr}} (Ae^{-t\mathcal{M}_q})
     &=\int\limits_{M}\sum_{\xi\in   \mathcal{I}} [\sigma(x,\xi)  \langle \xi \rangle^{-m}\langle \xi \rangle^{m} e^{-t\langle \xi \rangle^{q}}]\,u_\xi(x)\, \overline{v_\xi(x)}\,dx \nonumber\\
     &=\sum_{\xi\in   \mathcal{I}}\int\limits_{M}\lambda(x,\xi)\,u_\xi(x)\, \overline{v_\xi(x)}\,dx\times\Omega_{t}(\xi).
\end{align}
The hypothesis $\sigma\in {S}^{m}_{\rho,\delta}(M\times   \mathcal{I})$ implies that $|\lambda(x,\xi)|=O(1).$  
So,
\begin{align} \label{stepchange}
    |\textnormal{\bf {Tr}} (Ae^{-t\mathcal{M}_q})| &\leq   \sum_{\xi\in   \mathcal{I}}\int\limits_{M}|\lambda(x,\xi)||u_\xi(x)\, \overline{v_\xi(x)}|\,dx\times\Omega_{t}(\xi) \nonumber\\ &\leq \sum_{\xi\in   \mathcal{I}}  \Omega_{t}(\xi)\Vert u_\xi \Vert_{L^2}\Vert v_\xi \Vert_{L^2} =\sum_{\xi\in   \mathcal{I}}  \Omega_{t}(\xi),
\end{align}  where in the last line we have used that $\Vert u_\xi \Vert_{L^2} =1$ and $ \Vert v_\xi \Vert_{L^2}=1.$
Now, as above, we will use the Weyl-law for
$(1+L^{\circ}L)^{\frac{1}{2\nu}}$. Observe that,
\begin{align*}
  & \textnormal{\bf {Tr}}( (1+L^\circ L)^{\frac{m}{2\nu}} e^{-t(1+L^\circ L)^{\frac{q}{2\nu}}})\\
  \quad&=\sum_{k=0}^{\infty}\sum_{\xi:2^{k}\leqslant (1+|\lambda_{\xi}|^2)^\frac{1}{2\nu} <2^{k+1} } (1+|\lambda_{\xi}|^2)^\frac{m}{2\nu}e^{-t(1+|\lambda_{\xi}|^2)^{\frac{q}{2\nu}}      }.
\end{align*}Consequently,
\begin{align*}
   \textnormal{\bf {Tr}} (Ae^{-t\mathcal{M}_q})  &=\sum_{k=0}^{\infty}\sum_{\xi:2^{k}\leqslant (1+|\lambda_{\xi}|^2)^\frac{1}{2\nu} <2^{k+1}, } (1+|\lambda_{\xi}|^2)^\frac{m}{2\nu}e^{-t(1+|\lambda_{\xi}|^2)^{\frac{q}{2\nu}}      }\\
   &\asymp \sum_{k=0}^{\infty}
     \sum_{\xi:2^{k}\leqslant (1+|\lambda_{\xi}|^2)^\frac{1}{2\nu} <2^{ k+1} }  2^{km} e^{-t2^{kq}      }\\
    &\asymp\sum_{k=0}^{\infty} 2^{km} e^{-t2^{kq}}  N(2^{k})=\sum_{k=0}^{\infty} 2^{km} e^{-t2^{kq}} 2^{kQ}\\
    &=\sum_{k=0}^{\infty} e^{-t2^{kq}} 2^{(Q+m-1)k}2^{k}.
\end{align*}From the previous estimate and the fact that $q>0,$
\begin{align*}
   \sum_{k=0}^{\infty} e^{-t2^{kq}} 2^{(Q+m-1)k}2^{k}\asymp \int\limits_{1}^{\infty}e^{-t\lambda^{q}}\lambda^{Q+m-1}d\lambda=t^{-\frac{Q+m}{q}}\int\limits_{t^{\frac{1}{q}}}^{\infty}e^{-s^{q}}s^{Q+m-1}ds.
\end{align*}So, we have proved the first part of the theorem. Now, in particular, for $m=-Q,$ we have
\begin{align*}
 | \textnormal{\bf {Tr}}  (Ae^{-t\mathcal{M}_q})|\lesssim \int\limits_{t^{\frac{1}{q}}}^{\infty}e^{-s^{q}}s^{-1}ds. 
\end{align*}Observe that for $0<t<1,$ the main contribution in the integral $\int\limits_{t^{\frac{1}{q}}}^{\infty}e^{-s^{q}}s^{-1}ds$ is the integral of $G(s):=e^{-s^{q}}s^{-1},$ on the interval $[t^{\frac{1}{q}},1).$ Indeed, $\int\limits_{1}^{\infty}e^{-s^{q}}s^{-1}ds=o(1)$ for $q>0.$ Now, we can compute
\begin{align*}
 \int\limits_{t^{\frac{1}{q}}}^{1}e^{-s^{q}}s^{-1}ds  \sim  \int\limits_{t^{\frac{1}{q}}}^{1}s^{-1}ds=-\frac{1}{q}\log(t).
\end{align*}In the case $m>-Q,$ we have that the function $g(t)=\int\limits_{t}^{\infty}e^{-s^{q}}s^{Q+m -1}ds<\infty,$ is real analytic in $[0,\infty),$ and for $t\rightarrow 0^+,$ $g(t)=\sum_{k=0}^\infty b_{k}t^k,$ which implies 
\begin{align*}
  | \textnormal{\bf {Tr}}  (Ae^{-t\mathcal{M}_q})|\lesssim t^{-\frac{Q+m}{q}}\left(  \sum_{k=0}^{\infty}b_k't^{\frac{k}{q}}\right),\,\,t\rightarrow 0^{+} .
\end{align*}
So, we end the proof.
\end{proof}
\begin{remark}\label{remarkresidue}
Observe that we can summarise \eqref{asymp1212} and \eqref{asymp1212'} by writing 
\begin{equation}\label{summarisingformulas}
     |\textnormal{\bf {Tr}} (Ae^{-t(1+L^{\circ}L)^{ \frac{q}{2\nu} }  })|\leq C_{m,Q} t^{-\frac{m+Q}{q}}\sum_{k=0}^\infty a_kt^\frac{k}{q}-b_0\frac{1}{q}\log(t),\quad t\rightarrow0^{+},
 \end{equation}for  $m\geqslant    -Q.$ If $m=-Q,$ then $a_k=0$ for every $k,$ and for $m>-Q,$ $b_0=0.$
\end{remark}

\begin{remark}\label{remarkVishvesh} \label{remselfadj} It is worth noting that in Theorem \ref{asymptotictracemultiplier2}, we established only estimates for the traces due to lack of self-adjointness of the model operator $L.$ If $L$ is a self-adjoint operator (e.g., harmonic oscillators, anharmonic oscillators) then it is easy to provide asymptotic expansions of traces (with an additional assumption of the positivity of symbol $\sigma$) by following the same proof with use of self-adjointness of operator $L$ in a intermediate step. Indeed, if $L$ is self-adjoint then $u_\xi=v_\xi$ for $\xi\in \mathcal{I}.$ Thus, using the orthogonality of eigenfunctions of $L,$ from \eqref{lackofselfad} we obtain, instead of \eqref{stepchange}, the following asymptotic estimate
\begin{align} 
    \textnormal{\bf {Tr}} (Ae^{-t\mathcal{M}_q}) &=  \sum_{\xi\in   \mathcal{I}}\int\limits_{M}\lambda(x,\xi) u_\xi(x)\, \overline{u_\xi(x)} \,dx\times\Omega_{t}(\xi) \nonumber\\ &\asymp \sum_{\xi\in   \mathcal{I}}  \Omega_{t}(\xi) \int_M u_\xi(x) \overline{u_\xi(x)}\, dx =\sum_{\xi\in   \mathcal{I}}  \Omega_{t}(\xi),
\end{align}provided that $\sigma$ is $L$-elliptic, because in that case we have $\lambda(x,\xi)=o(1)$ (instead of the estimate $|\lambda(x,\xi)|=O(1)$ that we have  for a general symbol $\sigma\in S^{m}_{\rho,\delta}(M\times \mathcal{I})$ as in the proof of Theorem \ref{asymptotictracemultiplier2}).
Therefore, we get the following result on asymptotic expansions of traces. 
\begin{theorem}\label{L*=L1}Let $M=\overline{\Omega}$ be a smooth manifold with (possibly empty)  boundary $\partial \Omega$ and let $L$ be a self-adjoint operator. For $0\leqslant \rho, \delta \leqslant 1,$ let us consider an $L$-elliptic continuous linear operator $A:C^\infty_L(M)\rightarrow\mathcal{D}'_L(M)$ with positive symbol  $\sigma\in {S}^{m}_{\rho,\delta}(M\times   \mathcal{I})$. Let $\mathcal{M}_q:=(1+L^{\circ}L)^{\frac{q}{2\nu}}$ with $q>0$. Then, 
\begin{equation}
  \textnormal{\bf {Tr}}   (Ae^{-t\mathcal{M}_q}) \sim  c_{m,Q} t^{-\frac{Q+m}{q}}\int\limits_{t^{\frac{1}{q}}}^{\infty}e^{-s^{q}}s^{Q+m-1}ds,\,\,\forall t>0.
\end{equation}In particular, for $m=-Q,$ we have
\begin{equation}
   \textnormal{\bf {Tr}}  (Ae^{-t\mathcal{M}_q}) \sim  -c_Q\frac{1}{q}\log(t),\,\,\forall t\in (0,1),
\end{equation}while for $m>-Q,$ we have the asymptotic expansion
\begin{equation}
 \textnormal{\bf {Tr}}(Ae^{-t\mathcal{M}_q}) \sim C_{m,Q} t^{-\frac{Q+m}{q}}\left(  \sum_{k=0}^{\infty}b_k't^{\frac{k}{q}}\right),\,\,t\rightarrow 0^{+}.
\end{equation}
\end{theorem}
\end{remark}

\begin{example}Let us assume that $a(x)$ is an integrable positive function over $M.$ Let $P=a(x)A,$ where $A=\textnormal{Op}_L(\sigma) $ with $\sigma\in S^{m}_\rho( \mathcal{I}),$ $0\leqslant \rho\leqslant 1,$ is an operator of order $m\geqslant    -Q.$ Let us assume that $A$ is an $L$-elliptic operator and that $\sigma(\xi)\geq 0$ for all $\xi\in \mathcal{I}$. Because the $L^2$-trace of $P$ is the integral of its kernel on the diagonal, we have
\begin{align*}
   \textnormal{\bf {Tr}}(Pe^{-t(1+L^{\circ}L)^{ \frac{q}{2\nu} }  }) &=\int\limits_{M}\sum_{\xi\in  \mathcal{I}} (a(x)\sigma(\xi)e^{-t(1+|\lambda_\xi|^2)^{\frac{q}{2 \nu}}})dx\\
   &=\int\limits_{M}a(x)dx\times \textnormal{\bf {Tr}}  (Ae^{-t(1+L^{\circ}L)^{ \frac{q}{2\nu} }  }). 
\end{align*}This implies   an asymptotic expansion of the form
\begin{equation*} \textnormal{\bf {Tr}}(Pe^{-t(1+L^{\circ}L)^{ \frac{q}{2\nu} }  })\sim t^{-\frac{m+Q}{q}}\sum_{k=0}^\infty a_kt^\frac{k}{q}-b_0\frac{1}{q}\int\limits_{M}a(x)dx\log(t),\quad t\rightarrow0^{+}.
\end{equation*}Observe that if $\textnormal{Vol}(M):=\int_{M}dx<\infty,$ by taking $a\equiv 1,$ we have
\begin{align*}
   \textnormal{\bf {Tr}}(Pe^{-t(1+L^{\circ}L)^{ \frac{q}{2\nu} }  })\sim t^{-\frac{m+Q}{q}}\sum_{k=0}^\infty a_kt^\frac{k}{q}-\frac{b_0}{q}\textnormal{Vol}(M)\log(t),\quad t\rightarrow0^{+},
\end{align*}showing that these traces possess encoded geometric information of $M.$ 
\end{example}
Now, we will study regularised traces of the form $\textnormal{\bf{Tr}}(A\psi(t E))$ where $t\in \mathbb{R},$ $\psi$ is a compactly supported real-valued function and $E$ is an $L$-elliptic positive $L$-Fourier multiplier  of order $q>0.$

\begin{theorem}\label{asymptotictraceeta}
Let $M=\overline{\Omega}$ be a smooth manifold with (possibly empty)  boundary $\partial \Omega.$    For  $0\leqslant \rho,\delta \leqslant 1,$ let us consider a  continuous linear operator $A:C^\infty_L(M)\rightarrow\mathcal{D}'_L(M)$ with symbol  $\sigma\in {S}^{m}_{\rho,\delta}(M\times   \mathcal{I})$.   Let $E$ be a positive $L$-elliptic $L$-Fourier multiplier  of order $q>0$. Then 
\begin{equation}\label{asymp122222}
    |\textnormal{\bf{Tr}}(A\psi(t E))|\leq C_{Q} \frac{1}{q}\int\limits_{t^{\frac{1}{q}}}^{\infty}\psi(s)\times\frac{ds}{s},\,\,\forall t>0,
\end{equation}provided that $\psi\in L^{1}(\mathbb{R}^{+}_0;\frac{ds}{s})\cap C^{\infty}_0(\mathbb{R}^{+}_0)$ is positive, and $m=-Q.$ On the other hand, for $m>-Q$ and $\psi\in C^{\infty}_{0}(\mathbb{R}^+_0)$ being positive, we have
\begin{equation}\label{asymp1212222}
     |\textnormal{\bf{Tr}}(A\psi(t E))|\leq C_{m,Q} t^{-\frac{1}{q}(Q+m)}\frac{1}{q}\int\limits_{ t^{\frac{1}{q}}  }^\infty\psi(s)s^{\frac{Q+m}{q}}\times \frac{ds}{s},\,\,\forall t>0.
\end{equation}So, we have the estimate
\begin{equation}\label{asymp1212'22}
|\textnormal{\bf{Tr}}(A\psi(t E))|\leq t^{-\frac{Q+m}{q}}\left(  \sum_{k=0}^{\infty}a_kt^{\frac{k}{q}}\right),\,\,t\rightarrow 0^{+} .
\end{equation}for $m> -Q.$ 
\end{theorem}
\begin{proof}Let us denote by $\widehat{E}(\xi),\,\xi\in \mathcal{I},$ the $L$-symbol of $E.$
By writing the trace of $A\psi(t E)$ as the integral of its Schwartz kernel over  the diagonal, we have
 \begin{align*}
     | {\bf \textnormal{\bf {Tr}}}(A\psi(t E))|&\leq\int\limits_{M}\sum_{\xi\in   \mathcal{I}} |\sigma(x,\xi)\psi(t\widehat{E}(\xi))\,u_\xi(x) \overline{v_\xi(x)}|\,dx\\
     &=\int\limits_{M}\sum_{\xi\in   \mathcal{I}} |\sigma(x,\xi)  \langle \xi \rangle^{-m}\langle \xi \rangle^{m} \psi(t\widehat{E}(\xi))]\,u_\xi(x) \overline{v_\xi(x)}|\,dx\\
     &\leq  \sum_{\xi\in   \mathcal{I}} \int_M \,|u_\xi(x) \overline{v_\xi(x)}|\,dx\,  [\langle \xi \rangle^{m} \psi(t\widehat{E}(\xi))]\\
     &\leq \sum_{\xi\in   \mathcal{I}}\Vert u_\xi\Vert_{L^2}\Vert v_\xi\Vert_{L^2} |\langle \xi \rangle^{m} \psi(t\widehat{E}(\xi))|\\&= {\bf \textnormal{\bf {Tr}}}(\langle \xi \rangle^{m} \psi(t\widehat{E}(\xi))),
\end{align*}where we have used  that $\sigma\in {S}^{m}_{\rho,\delta}(M\times   \mathcal{I})$ to deduce $\Vert  \sigma(\cdot,\cdot)  \langle \cdot \rangle^{-m}\Vert_{L^{\infty}(M\times \mathcal{I})}<\infty$.

Again, we will use the Weyl-law for $(1+L^{\circ}L)^{\frac{1}{2\nu}}$. Observe that by following a similar analysis as in \eqref{ellipticityofE}, the $L$-ellipticity of $E,$ which has order $q,$ and its positivity, imply
$$ t \widehat{E}(\xi)\sim t(1+|\lambda_{\xi}|^2)^\frac{q}{2\nu} .$$
So, we get,
\begin{align*}
  & \textnormal{\bf {Tr}}(\langle \xi \rangle^{m} \psi(t\widehat{E}(\xi)))\\
  \quad&=\sum_{k=0}^{\infty}\sum_{\xi:2^{k}\leqslant (1+|\lambda_{\xi}|^2)^\frac{1}{2\nu} <2^{k+1} } (1+|\lambda_{\xi}|^2)^\frac{m}{2\nu}\psi(t\widehat{E}(\xi))\\
  \quad&\asymp\sum_{k=0}^{\infty}\sum_{\xi:2^{k}\leqslant (1+|\lambda_{\xi}|^2)^\frac{1}{2\nu} <2^{k+1} } (1+|\lambda_{\xi}|^2)^\frac{m}{2\nu}\psi\left(t(1+|\lambda_{\xi}|^2)^\frac{q}{2\nu}\right)\\
  &\asymp \sum_{k=0}^{\infty}\sum_{\xi:2^{k}\leqslant (1+|\lambda_{\xi}|^2)^\frac{1}{2\nu} <2^{k+1}} 2^{km}\psi(t2^{kq}) , 
\end{align*}and consequently,
\begin{align*}
   \textnormal{\bf {Tr}}  (\langle \xi \rangle^{m} \psi(t\widehat{E}(\xi)))&\asymp \sum_{k=0}^{\infty} 2^{km}\psi(t2^{kq}) \sum_{\xi:2^{k}\leqslant (1+|\lambda_{\xi}|^2)^\frac{1}{2\nu} <2^{k+1}} \\
    &\asymp\sum_{k=0}^{\infty} 2^{km}\psi(t2^{kq}) N(2^{k})\asymp\sum_{k=0}^{\infty} 2^{km}\psi(t2^{kq}) 2^{kQ}\\
    &=\sum_{k=0}^{\infty} \psi(t2^{kq}) 2^{k(Q +m -1)}2^{k}.
\end{align*}
Estimating the sums in $k$ as a integral, we have
\begin{align*}
  & \sum_{k=0}^{\infty} \psi(t2^{kq}) 2^{k(Q +m -1)}2^{k}\asymp \int\limits_{0}^{\infty}\psi(t\lambda^{q})\lambda^{Q +m -1}d\lambda\\
   &= t^{-\frac{1}{q}(Q+m)} \frac{1}{q} \times \int\limits_{  t^{\frac{1}{q}}  }^\infty\psi(s)s^{\frac{Q+m}{q}}\frac{ds}{s}.
\end{align*}In particular, for $m=-Q,$ we have
\begin{align*}
   \textnormal{\bf {Tr}} (\langle \xi \rangle^{m} \psi(t\widehat{E}(\xi))\sim \frac{1}{q}\int\limits_{t^{\frac{1}{q}}}^{\infty}\psi(s)\frac{ds}{s},
\end{align*} provided that the compactly supported function $\psi$ on $\mathbb{R}^+_0$ belongs to $L^1(\mathbb{R}_0^+,\frac{ds}{s}).$
Observe that the integral $\int\limits_{0}^\infty\psi(s)s^{\frac{Q+m}{q}}\times \frac{ds}{s}$ makes sense if $\psi$  is smooth and it has compact support in $(0,\infty).$ However if $\psi(0)\neq 0,$ in order to assure that $$\int\limits_{0}^\infty\psi(s)s^{\frac{Q+m}{q}}\times \frac{ds}{s}<\infty,\,\,\,\psi\in C^{\infty}_0(\mathbb{R}^+_0),$$ we require the condition $1-\frac{Q+m}{q}<1,$ or equivalently that $Q+m> 0.$  So, in  such a situation, the function $$G(s):=s^{\frac{Q+m}{q}} \textnormal{\bf {Tr}}(\textnormal{Op}_L(\langle \xi \rangle^{m} \psi(s\widehat{E}(\xi))),\quad s>0,$$ is real-analytic and we can deduce the estimate  \eqref{asymp1212'22}. Thus, we end the proof.
\end{proof}

\begin{remark}
Similar to Remark \ref{remarkVishvesh},  under the condition of self-adjointness of $L,$ we have the following result on asymptotic expansion of  regularized traces by performing the appropriate changes in the proof as discussed in Remark \ref{remarkVishvesh}. The $L$-ellipticity of the symbol $\sigma$ and its positivity will be required in order that $\lambda(x,\xi)=o(1),$ with $\lambda(x,\xi)$ defined in the proof of Theorem \ref{asymptotictraceeta}.
\begin{theorem}\label{L*=L2}
Let $M=\overline{\Omega}$ be a smooth manifold with (possibly empty)  boundary $\partial \Omega$ and let $L$ be a self-adjoint operator. For  $0\leqslant \rho,\delta \leqslant 1,$ let us consider a   $L$-elliptic continuous linear operator $A:C^\infty_L(M)\rightarrow\mathcal{D}'_L(M)$ with positive symbol  $\sigma\in {S}^{m}_{\rho,\delta}(M\times   \mathcal{I})$.   Let $E$ be a positive $L$-elliptic $L$-Fourier multiplier  of order $q>0$. Then 
\begin{equation}
    \textnormal{\bf{Tr}}(A\psi(t E)) \sim C_{Q} \frac{1}{q}\int\limits_{t^{\frac{1}{q}}}^{\infty}\psi(s)\times\frac{ds}{s},\,\,\forall t>0,
\end{equation}provided that $\psi\in L^{1}(\mathbb{R}^{+}_0;\frac{ds}{s})\cap C^{\infty}_0(\mathbb{R}^{+}_0)$ is positive, and $m=-Q.$ On the other hand, for $m>-Q$ and $\psi\in C^{\infty}_{0}(\mathbb{R}^+_0)$ being positive, we have
\begin{equation}
     \textnormal{\bf{Tr}}(A\psi(t E)) \sim C_{m,Q} t^{-\frac{1}{q}(Q+m)}\frac{1}{q}\int\limits_{ t^{\frac{1}{q}}  }^\infty\psi(s)s^{\frac{Q+m}{q}}\times \frac{ds}{s},\,\,\forall t>0.
\end{equation}So, we have the asymptotic expansion
\begin{equation}\label{321}
\textnormal{\bf{Tr}}(A\psi(t E)) \sim t^{-\frac{Q+m}{q}}\left(  \sum_{k=0}^{\infty}a_kt^{\frac{k}{q}}\right),\,\,t\rightarrow 0^{+} .
\end{equation}for $m> -Q.$ 
\end{theorem}
\end{remark}

\section{Dixmier tracebility of $L$-elliptic operators}\label{Dixmiersection}
In this section, we consider the Dixmier tracebility of the $L$-elliptic  global pseudo-differential operators on a compact manifold with  boundary. One of the main tool in our analysis is the global functional calculus developed by the authors in \cite{CardonaKumarRuzhanskyTokmagambetov2020II}, for which  we require the WZ-condition.

Let us briefly recall the definition of the Dixmier ideal  $\mathscr{L}^{(1,\infty)}(H)$ on a Hilbert space $H$ (we are interested in $H=L^2(M)$ for instance) as in \cite{Connes94} (see also \cite{Sukochev}). If we denote the sequence of singular values of $A$,  i.e. the square roots of the eigenvalues of the non-negative self-adjoint operator $A^\ast A$ by $\{s_{n}(A)\}$ then 

$$\mathscr{L}^{(1,\infty)}(H):=\Big\{A\in\mathscr{L}(H): A\, \textnormal{is compact, and}\,\,\sum_{1\leq n\leq N}s_{n}(A)=O(\log(N)),\,\,N\rightarrow \infty  \Big\}.$$
So, $\mathscr{L}^{(1,\infty)}(H)$ is endowed with the norm
\begin{equation}\label{dixmier}
\Vert A \Vert_{\mathcal{L}^{(1,\infty)}(H)}=\sup_{N\geq 2}\frac{1}{\log(N)}\sum_{1\leq n\leq N}s_{n}(A).
\end{equation}Let us define the functional $\textnormal{\bf {Tr}}_\omega: \mathcal{L}^{(1,\infty)}(H)\rightarrow (0,\infty],$ given by
\begin{equation}\label{functional}
\textnormal{\bf {Tr}}_\omega(A):=\lim_{N\rightarrow \infty}\frac{1}{\log(N)}\sum_{1\leq n\leq N}s_{n}(A).
\end{equation}
Our first result in this direction is a characterisation of operators belonging to the Dixmier ideal in terms of order of the operators. 
\begin{lemma}\label{lemmadixmier} Let $M=\overline{\Omega}$ be a smooth manifold with (possibly empty)  boundary $\partial \Omega.$
  For  $0\leqslant \rho\leqslant 1,$ let us consider a positive  $L$-elliptic continuous operator $A:C^\infty(M)\rightarrow\mathscr{D}'(M)$ with symbol $\sigma$ depending on $\xi$ only, such that  $\sigma\in {S}^{m}_{\rho}( \mathcal{I})$, $m\in \mathbb{R}$.  Then, $A$ belongs to the Dixmier ideal $\mathscr{L}^{1,\infty}(L^2(M))$ with $\textnormal{\bf{Tr}}_{w}(A)<\infty,$ if and only if $m\leqslant -Q.$ If $A\neq 0,$ $\textnormal{\bf{Tr}}_{w}(A)\asymp\frac{1}{Q}$ for $m=-Q,$ and for $m<-Q,$ $\textnormal{\bf{Tr}}_{w}(A)=0.$ Moreover, for $m=-Q,$ the constants of proportionality in $\textnormal{\bf{Tr}}_{w}(A)\asymp\frac{1}{Q}$ only depend of the  norm $\Vert \sigma(\cdot) \langle \cdot\rangle^{-m} \Vert_{L^{\infty}(\mathcal{I})}.$
  \end{lemma}
\begin{proof}
Let us use the positivity of $A$ for computing the Dixmier trace of $A$ using the Tauberian
theorem of Hardy and Littlewood  (see e.g. \cite[Proposition 4, Page 313]{Connes94})
\begin{equation*}
\textnormal{\bf{Tr}}_{w}(A)=\lim_{p\rightarrow 1^{+}}    (p-1)\textnormal{\bf{Tr}}(A^p).
\end{equation*}
 The spectral mapping theorem implies that
\begin{equation*}
    \textnormal{spectrum}(A^p)=\{{\sigma(\xi)^p}:  \xi\in \mathcal{I}\}.
    \end{equation*}So, we have
    \begin{align*}
        \textnormal{\textbf{Tr}}(A^p)=\sum_{\xi \in  \mathcal{I}}{\sigma(\xi)}^p.\,\,\,\,
    \end{align*}Using the functional calculus in \cite{CardonaKumarRuzhanskyTokmagambetov2020II}, and the $L$-ellipticity of $\textnormal{Op}_L(\sigma^p),$ we have $\sigma^p\in {S}^{mp}_{\rho}( \mathcal{I})$, $m\in \mathbb{R},$ and for all $p>1,$
    in view of \eqref{EQ:symbol-class}, we have that
\begin{equation*}
  \left|\sigma(\xi)^p \right|
        \leq C_{\sigma, mp}
                 \langle\xi\rangle^{mp},\,\,\xi\in \mathcal{I}.
\end{equation*} The previous estimate allows us to conclude that $\Vert \sigma(\cdot)^{p} \langle \cdot\rangle^{-mp}  \Vert_{L^{\infty}(\mathcal{I})}<\infty.$ So, we have,
\begin{align*}
     \textnormal{\textbf{Tr}}(A^p)&=\sum_{ \xi\in  \mathcal{I}}{\sigma(\xi)^p (1+|\lambda_\xi|^2)^{-\frac{pm}{2\nu}}(1+|\lambda_\xi |^2)^{\frac{mp}{2\nu}}      }\\
     &\asymp \Vert \sigma(\cdot)^{p} \langle \cdot\rangle^{-mp}  \Vert_{L^{\infty}(\mathcal{I})} \sum_{\xi\in  \mathcal{I}}  (1+|\lambda_\xi|^2)^{\frac{mp}{2\nu}} \\
     &= \Vert \sigma(\cdot)^{p} \langle \cdot\rangle^{-mp} \Vert_{L^{\infty}(\mathcal{I})}\textnormal{\textbf{Tr}}((1+L^\circ L)^{\frac{mp}{2\nu}}).
\end{align*}
Now, as in the previous section, we will use the Weyl-law for $(1+L^\circ L)^{\frac{1}{2\nu}}$. Observe that,
\begin{align*}
  \textnormal{\textbf{Tr}}((1+L^\circ L)^{\frac{mp}{2\nu }})&=\sum_{k=0}^{\infty}\sum_{\xi :2^{k}\leqslant (1+|\lambda_\xi|^2)^\frac{1}{2\nu} <2^{k+1}} {(1+|\lambda_\xi|^2)^{\frac{mp}{2\nu}}      }\\ &\asymp \sum_{k=0}^{\infty} {2^{kmp}}N(2^{k})\asymp \sum_{k=0}^{\infty} 2^{kmp}2^{kQ}=\sum_{k=0}^{\infty} {2^{kmp}}2^{k(Q-1)}2^{k}.
\end{align*}Observe that 
\begin{align*}
    \sum_{k=0}^{\infty} 2^{k(Q+mp-1)}2^{k}\asymp \int\limits_{1}^{\infty}\lambda^{Q+mp-1}d\lambda<\infty,
\end{align*} for all $p>1,$ if and only if $m\leqslant  -Q.$ So, from the identity
\begin{equation*}
    \int\limits_{1}^{\infty}\lambda^{Q+mp-1}d\lambda =-\frac{1}{Q+mp},
\end{equation*}
we deduce that
\begin{equation*}
    \textnormal{\bf{Tr}}_{w}(A)\asymp \lim_{p\rightarrow 1^{+}}(p-1)\times \frac{-1}{Q+mp}=\delta_{m,-Q}\times \frac{1}{Q},\,\,\,m\leqslant -Q, 
\end{equation*} with a constant of proportionality depending on  $$ \Vert \sigma(\cdot) \langle \cdot\rangle^{-m}  \Vert_{L^{\infty}(\mathcal{I})}=\lim_{p\rightarrow1^{+}} \Vert \sigma(\cdot)^{p} \langle \cdot\rangle^{-mp}  \Vert_{L^{\infty}(\mathcal{I})},$$ where $\delta_{m,-Q}$ is the Kronecker delta.
Thus, we end the proof.
\end{proof}

Using the above lemma we can deduce the following interesting result for general continuous linear operators.

\begin{corollary}\label{beautifulproof}
For  $0\leqslant\delta< \rho\leqslant 1,$   let us consider a continuous linear operator $A:C^\infty(M)\rightarrow\mathscr{D}'(M)$ with symbol  $\sigma\in {S}^{m}_{\rho,\delta}( M\times \mathcal{I})$, with $m<-Q.$ Then $\textnormal{\bf{Tr}}_{w}(A)=0.$
\end{corollary}
\begin{proof}
    We will use the notation $s_{n}(T),$ $n\in \mathbb{N}_0,$ for the sequence of singular values of a compact operator $T$ on a Hilbert space $H.$ Then, the following inequality holds (see \cite[Page 75]{Bathia}): $s_{n}(CB)\leq \Vert C\Vert_{\textnormal{Op}_L}s_{n}(B),$ for $C$ a bounded linear operator and $B$ a compact linear operator. From the definition of the functional $\textnormal{\bf{Tr}}_{w},$ we conclude easily that $0\leq \textnormal{\bf{Tr}}_{w}(CB)\leq \Vert C\Vert_{\textnormal{Op}_L}\textnormal{\bf{Tr}}_{w}(B) .$ Now, let us use this inequality in our setting. From the Calder\'on-Vaillancourt Theorem (Theorem \ref{L2}), we have that $A\mathcal{M}^{-m}\in  {S}^{0}_{\rho,\delta}( M\times \mathcal{I})$ extends to a bounded operator on $L^2(M),$ where $\mathcal{M}:=(1+L^\circ L)^{\frac{1}{2\nu}}.$ Consequently, 
    \begin{align*}
    0\leq \textnormal{\bf{Tr}}_{w}(A\mathcal{M}^{-m}\mathcal{M}^{m})\leq \Vert A\mathcal{M}^{-m}\Vert_{\mathscr{B}(L^2(G))}\textnormal{\bf{Tr}}_{w}(\mathcal{M}^{m})=0,    
    \end{align*}where we have used that $\textnormal{\bf{Tr}}_{w}(\mathcal{M}^{m})=0$ in view of Lemma \ref{lemmadixmier}. This implies that $\textnormal{\bf{Tr}}_{w}(A)=0.$ The proof is completed.
\end{proof}

The following theorem is the main result of this section. In particular, we extend the main result of  \cite{CardonaDelCorral} regarding the non-harmonic case.
\begin{theorem}\label{Dtrace} Let $M=\overline{\Omega}$ be a smooth manifold with (possibly empty)  boundary $\partial \Omega.$
  For  $0\leqslant\delta< \rho\leqslant 1,$    let us consider an  $L$-elliptic continuous linear operator $A:C^\infty(M)\rightarrow\mathscr{D}'(M)$ with symbol  $\sigma\in {S}^{m}_{\rho,\delta}( M\times \mathcal{I})$, $m\in \mathbb{R} $.  Let us assume that  $\sigma (x,\xi) \geqslant     0,$ for every $(x, \xi)\in M\times \mathcal{I}.$  If $A\neq 0,$ $|\textnormal{\bf{Tr}}_{w}(A)|\asymp\frac{1}{Q}$ for $m=-Q,$ and $\textnormal{\bf{Tr}}_{w}(A)=0$ for $m<-Q$. 
\end{theorem}
\begin{proof}
    For every $z\in M,$ let us consider the Fourier multiplier associated to the symbol $\sigma(z,\cdot),$ $A_z$ which satisfies the hypothesis in Lemma \ref{lemmadixmier}. Indeed, $\sigma (z, \xi)\geqslant     0,$ for every $\xi$ implies that $A_{z}$ is also positive and the $L$-ellipticity of $A$ implies the $L$-ellipticity of $A_{z}$ for every $z\in M.$ Observe that from the functional calculus for boundary value problems developed in \cite{CardonaKumarRuzhanskyTokmagambetov2020II}, and the $L$-ellipticity of $A,$ we have that
    \begin{equation}
        A^p=\textnormal{Op}_L[(x,\xi)\mapsto\sigma(x,\xi)^{p}]+R_{p},
    \end{equation}where $R_{p}$ is a pseudo-differential operator of order $mp-(\rho-\delta).$  Because $m\leqslant -Q,$ and $p\rightarrow 1^+,$ from Corollary \ref{beautifulproof}, we deduce that $\textnormal{\bf{Tr}}_w(R_p)=0.$ Indeed, the order of $R_p$ is $mp-(\rho-\delta)<-Q.$ So, note that
    \begin{align*}
        \textnormal{\bf{Tr}}_w(A^p)&=\textnormal{\bf{Tr}}_w(\textnormal{Op}_L[(x,\xi)\mapsto\sigma(x,\xi)^{p}])+\textnormal{\bf{Tr}}_w(R_{p})\\&=\textnormal{\bf{Tr}}_w(\textnormal{Op}_L[(x,\xi)\mapsto\sigma(x,\xi)^{p}]).
    \end{align*}
Observe that, integrating over the diagonal of the Schwartz kernel of $\textnormal{Op}_L[(x,\xi)\mapsto\sigma(x,\xi)^{p}],$ (see \cite{DRT2017} for details), we have
 \begin{align*}
 &|\textnormal{\bf{Tr}}_w(\textnormal{Op}_L[(x,\xi)\mapsto\sigma(x,\xi)^{p}])|=\left|\lim_{p\rightarrow 1^{+}}(p-1)    \int\limits_{M}\sum_{\xi\in  \mathcal{I}} \sigma(z,\xi)^p u_\xi(z) \overline{v_\xi(z)} dz\right|\\
 &\leq     \int\limits_{M}\lim_{p\rightarrow 1^{+}}(p-1)\sum_{\xi\in  \mathcal{I}} \sigma(z,\xi)^p |u_\xi(z) \overline{v_\xi(z)}| dz\\
 &\leq    \sup_{y\in M}\left(\lim_{p\rightarrow 1^{+}}(p-1)\sum_{\xi\in  \mathcal{I}} \sigma(y,\xi)^p \right)\sup_{\eta\in  \mathcal{I}}\int\limits_{M} |u_\eta(z) \overline{v_\eta(z)}| dz\\
 &\leq \sup_{y\in M}\textnormal{\textbf{Tr}}_w[\textnormal{Op}_L[\xi\mapsto\sigma(y,\xi)]] \sup_{\eta\in  \mathcal{I}}\Vert u_\eta\Vert_{L^2}\Vert v_\eta\Vert_{L^2}\\
 &\asymp  \sup_{y\in M} \Vert \sigma(y,\cdot) \langle \cdot\rangle^{-m}  \Vert_{L^{\infty}(\mathcal{I})}\frac{1}{Q}\delta_{m,-Q},
 \end{align*} where we have used (in virtue of Lemma \ref{lemmadixmier}) that, for any $y\in M,$ the Dixmier functional  $\textnormal{\textbf{Tr}}_w[\textnormal{Op}_L[\xi\mapsto\sigma(y,\xi)]] $ is proportional to the norm $\Vert \sigma(y,\cdot) \langle \cdot\rangle^{-m} \sigma \Vert_{L^{\infty}(\mathcal{I})}$. Observing that
 $$  \sup_{y\in M} \Vert \sigma(y,\cdot) \langle \cdot\rangle^{-m}  \Vert_{L^{\infty}(\mathcal{I})}=\Vert \sigma(\cdot,\cdot) \langle \cdot\rangle^{-m} \Vert_{L^{\infty}(M\times \mathcal{I})}<\infty,  $$
 we end the proof.
\end{proof}

\section{Appendix: ${L}$-Convolution, Plancherel formula and Sobolev spaces} 
\label{SEC:conv}

Let us introduce a notion of the  ${L}$-convolution, an analogue of the convolution adapted to the
boundary problem ${L}_\Omega$.

\begin{definition} (${L}$-Convolution) \label{Convolution}
For $f, g\in C_{{L}}^{\infty}(\overline{\Omega})$ define their ${L}$-convolution
by
\begin{equation}\label{EQ: CONV1}
(f\sL g)(x):=\sum_{\xi\in\mathcal{I}}\widehat{f}(\xi)\widehat{g}(\xi)u_{\xi}(x).
\end{equation}
By Proposition \ref{LEM: FTinS}  it is well-defined and we have
$f\sL g\in C_{{L}}^{\infty}(\overline{\Omega}).$

Moreover, due to the rapid decay of $L$-Fourier coefficients of
functions in $C_{{L}}^{\infty}(\overline{\Omega})$ compared to
a fixed polynomial growth of elements of $\mathcal S'(\mathcal{I})$, the
definition \eqref{EQ: CONV1}  still makes sense if $f\in \mathcal
D^\prime_{L}(\overline{\Omega})$ and $g\in C_{{
L}}^{\infty}(\overline{\Omega})$, with $f\sL g\in C_{{
L}}^{\infty}(\overline{\Omega}).$
\end{definition}

Analogously to the ${L}$-convolution, we can introduce the ${L}^*$-convolution.
Thus, for $f, g\in C_{{
L^{\ast}}}^{\infty}(\overline{\Omega})$, we define the ${
L^{\ast}}$-convolution using the ${L}^*$-Fourier transform by
\begin{equation}\label{EQ: CONV2}
(f\sLs g)(x):=\sum_{\xi\in\mathcal{I}}\widehat{f}_{\ast}(\xi)\widehat{g}_{\ast}(\xi)v_{\xi}(x).
\end{equation}
Its properties are similar to those of the ${L}$-convolution, so we may
formulate only the latter.

\begin{remark}
Informally, expanding the definitions of the Fourier transforms in \eqref{EQ: CONV1},
we can also write
\begin{equation}
\label{CONV} (f\sL
g)(x):=\int\limits_{\Omega}\int\limits_{\Omega}F(x,y,z)f(y)g(z)dydz,
\end{equation}
where
$$
F(x,y,z):=\sum_{\xi\in\mathcal{I}}u_{\xi}(x) \ \overline{v_{\xi}(y)}
\ \overline{v_{\xi}(z)}.
$$
The latter series should be understood in the sense of distributions. 

\medskip

\end{remark}

\begin{proposition}\label{ConvProp}For any $f, g\in C_{{
L}}^{\infty}(\overline{\Omega})$ we have
$$\widehat{f\sL g}=\widehat{f}\times \widehat{g},\,\xi\in \mathcal{I}.$$
The convolution is commutative and associative.
If $g \in C_{{
L}}^{\infty}(\overline{\Omega}),$ then for all
$f\in \mathcal D^\prime_{L}(\overline{\Omega})$ we have
\begin{equation}\label{EQ:conv1}
f\sL g\in C_{{L}}^{\infty}(\overline{\Omega}).
\end{equation}
In addition,  if $\Omega\subset \mathbb{R}^n$ is bounded, and $f,g\in  L^{2}(\overline{\Omega})$, then $f\sL g\in L^{1}(\overline{\Omega})$ with
$$\|f\sL g\|_{L^1}\leq C|\Omega|^{1/2} \|f\|_{L^2}\|g\|_{L^2},$$
where $|\Omega|$ is the volume of $\Omega$, with $C$ independent
of $f,g,\Omega$.
\end{proposition}

\subsection{Plancherel formula, Sobolev spaces  and their Fourier images}
\label{SEC:Sobolev}

In this subsection we discuss Sobolev spaces adapted to ${L}_\Omega$ and their images
under the $L$-Fourier transform. We start with the $L^2$-setting, where we can recall inequalities
between $L^2$-norms of functions and sums of squares of their Fourier coefficients. However, below we show that we actually have the Plancherel
identity in a suitably defined space $l^{2}_{{L}}$ and its conjugate $l^{2}_{{L}^{*}}$.

\medskip
Let us denote by $$l^{2}_{{L}}=l^2({L})$$ 
the linear space of complex-valued functions $a$
on $\mathcal{I}$ such that $\mathcal F^{-1}_{{L}}a\in
L^{2}(\overline{\Omega})$, i.e. if there exists $f\in L^{2}(\overline{\Omega})$ such that $\mathcal F_{{L}}f=a$.
Then the space of sequences $l^{2}_{{L}}$ is a
Hilbert space with the inner product
\begin{equation}\label{EQ: InnerProd SpSeq-s}
(a,\ b)_{l^{2}_{{
L}}}:=\sum_{\xi\in\mathcal{I}}a(\xi)\ \overline{(\mathcal F_{{
L^{\ast}}}\circ\mathcal F^{-1}_{{L}}b)(\xi)}
\end{equation}
for arbitrary $a,\,b\in l^{2}_{{L}}$.
The norm of $l^{2}_{{L}}$ is then given by the
formula
\begin{equation}\label{EQ:l2norm}
\|a\|_{l^{2}_{{L}}}=\left(\sum_{\xi\in\mathcal{I}}a(\xi)\
\overline{(\mathcal F_{{L^{\ast}}}\circ\mathcal F^{-1}_{{
L}}a)(\xi)}\right)^{1/2}, \quad \textrm{ for all } \, a\in l^{2}_{{L}}.
\end{equation} Analogously, we introduce the
Hilbert space $ l^{2}_{{L^{\ast}}}=l^{2}({L^{\ast}})$ 
as the space of functions $a$ on $\mathcal{I}$
such that $\mathcal F^{-1}_{{L^{\ast}}}a\in L^{2}(\overline{\Omega})$,
with the inner product
\begin{equation}
\label{EQ: InnerProd SpSeq-s_2} (a,\ b)_{l^{2}_{{
L^{\ast}}}}:=\sum_{\xi\in\mathcal{I}}a(\xi)\ \overline{(\mathcal
F_{{L}}\circ\mathcal F^{-1}_{{L^{\ast}}}b)(\xi)}
\end{equation}
for arbitrary $a,\,b\in l^{2}_{{L^{\ast}}}$. The norm of
$l^{2}_{{L^{\ast}}}$ is given by the formula
$$
\|a\|_{l^{2}_{{L^{\ast}}}}=\left(\sum_{\xi\in\mathcal{I}}a(\xi)\
\overline{(\mathcal F_{{L}}\circ\mathcal F^{-1}_{{
L^{\ast}}}a)(\xi)}\right)^{1/2}
$$
for all $a\in l^{2}_{{L^{\ast}}}$. The spaces of sequences
$l^{2}_{{L}}$ and
$l^{2}_{{L^{\ast}}}$ are thus generated by biorthogonal systems
$\{u_{\xi}\}_{\xi\in\mathcal{I}}$ and $\{v_{\xi}\}_{\xi\in\mathcal{I}}$.
The reason for their definition in the above forms becomes clear again
in view of the following Plancherel identity:

\begin{proposition} {\rm(Plancherel's identity)}\label{PlanchId}
If $f,\,g\in L^{2}(\overline{\Omega})$ then
$\widehat{f},\,\widehat{g}\in l^{2}_{{L}}, \,\,\,
\widehat{f}_{\ast},\, \widehat{g}_{\ast}\in l^{2}_{{\rm
L^{\ast}}}$, and the inner products {\rm(\ref{EQ: InnerProd SpSeq-s}),
(\ref{EQ: InnerProd SpSeq-s_2})} take the form
$$
(\widehat{f},\ \widehat{g})_{l^{2}_{{L}}}=\sum_{\xi\in\mathcal{I}}\widehat{f}(\xi)\ \overline{\widehat{g}_{\ast}(\xi)}
$$
and
$$
(\widehat{f}_{\ast},\ \widehat{g}_{\ast})_{l^{2}_{{
L^{\ast}}}}=\sum_{\xi\in\mathcal{I}}\widehat{f}_{\ast}(\xi)\
\overline{\widehat{g}(\xi)}.
$$
In particular, we have
$$
\overline{(\widehat{f},\ \widehat{g})_{l^{2}_{{L}}}}=
(\widehat{g}_{\ast},\ \widehat{f}_{\ast})_{l^{2}_{{
L^{\ast}}}}.
$$
The Parseval identity takes the form
\begin{equation}\label{Parseval}
(f,g)_{L^{2}}=(\widehat{f},\widehat{g})_{l^{2}_{{
L}}}=\sum_{\xi\in\mathcal{I}}\widehat{f}(\xi)\ \overline{\widehat{g}_{\ast}(\xi)}.
\end{equation}
Furthermore, for any $f\in L^{2}(\overline{\Omega})$, we have
$\widehat{f}\in l^{2}_{{L}}$, $\widehat{f}_{\ast}\in l^{2}_{{
L^{\ast}}}$, and
\begin{equation}
\label{Planch} \|f\|_{L^{2}}=\|\widehat{f}\|_{l^{2}_{{
L}}}=\|\widehat{f}_{\ast}\|_{l^{2}_{{L^{\ast}}}}.
\end{equation}
\end{proposition}

Now we introduce Sobolev spaces generated by the operator ${L}_{\Omega}$:

\begin{definition}[Sobolev spaces $\mathcal \mathcal{H}^{s}_{{L}}(\overline{\Omega})$] \label{SobSp}
For $f\in\mathcal D'_{{L}}(\overline{\Omega})\cap \mathcal D'_{{L}^{*}}(\overline{\Omega})$ and $s\in\mathbb R$, we say that
$$f\in\mathcal \mathcal{H}^{s}_{{L}}(\overline{\Omega})\, {\textrm{ if and only if }}\,
\langle\xi\rangle^{s}\widehat{f}(\xi)\in l^{2}_{{L}}.$$
We define the norm on $\mathcal \mathcal{H}^{s}_{{L}}(\overline{\Omega})$ by
\begin{equation}\label{SobNorm}
\|f\|_{\mathcal \mathcal{H}^{s}_{{
L}}(\overline{\Omega})}:=\left(\sum_{\xi\in\mathcal{I}}
\langle\xi\rangle^{2s}\widehat{f}(\xi)\overline{\widehat{f}_{\ast}(\xi)}\right)^{1/2}.
\end{equation}
The Sobolev space $\mathcal \mathcal{H}^{s}_{{L}}(\overline{\Omega})$ is then the
space of ${L}$-distributions $f$ for which we have
$\|f\|_{\mathcal \mathcal{H}^{s}_{{L}}(\overline{\Omega})}<\infty$. Similarly,
we can define the
space $\mathcal \mathcal{H}^{s}_{{L^{\ast}}}(\overline{\Omega})$ by the
condition
\begin{equation}\label{SobNorm2}
\|f\|_{\mathcal \mathcal{H}^{s}_{{
L^{\ast}}}(\overline{\Omega})}:=\left(\sum_{\xi\in\mathcal{I}}\langle\xi\rangle^{2s}\widehat{f}_{\ast}(\xi)\overline{\widehat{f}(\xi)}\right)^{1/2}<\infty.
\end{equation}
\end{definition}
We note that the expressions in \eqref{SobNorm} and 
\eqref{SobNorm2} are well-defined since the sum
$$
\sum_{\xi\in\mathcal{I}}
\langle\xi\rangle^{2s}\widehat{f}(\xi)\overline{\widehat{f}_{\ast}(\xi)}=
(\langle\xi\rangle^{s}\widehat{f}(\xi),\langle\xi\rangle^{s}\widehat{f}(\xi))_{l^{2}_{L}}\geq 0
$$
is real and non-negative. 
Consequently, since we can write the sum in \eqref{SobNorm2} as the
complex conjugate of that in  \eqref{SobNorm}, and with both being real,
we see that the spaces $\mathcal \mathcal{H}^{s}_{{L}}(\overline{\Omega})$ and 
$\mathcal \mathcal{H}^{s}_{{L^{\ast}}}(\overline{\Omega})$ coincide as sets. Moreover, we have

\begin{proposition}\label{SobHilSpace}
For every $s\in\mathbb R$, the Sobolev space
$\mathcal \mathcal{H}^{s}_{{L}}(\overline{\Omega})$ is a Hilbert space with the
inner product
$$
(f,\ g)_{\mathcal \mathcal{H}^{s}_{{L}}(\overline{\Omega})}:=\sum_{\xi\in\mathcal{I}
}\langle\xi\rangle^{2s}\widehat{f}(\xi)\overline{\widehat{g}_{\ast}(\xi)}.
$$
Similarly,
the Sobolev space
$\mathcal \mathcal{H}^{s}_{{L^{\ast}}}(\overline{\Omega})$ is a Hilbert space with
the inner product
$$
(f,\ g)_{\mathcal \mathcal{H}^{s}_{{
L^{\ast}}}(\overline{\Omega})}:=\sum_{\xi\in\mathcal{I}}\langle\xi\rangle^{2s}\widehat{f}_{\ast}(\xi)\overline{\widehat{g}(\xi)}.
$$
For every $s\in\mathbb R$, the Sobolev spaces  $\mathcal \mathcal{H}^{s}_{{L}}(\overline{\Omega})$,
and $\mathcal \mathcal{H}^{s}_{{L}^*}(\overline{\Omega})$ are
isometrically isomorphic.
\end{proposition}

\bibliographystyle{amsplain}

\end{document}